\documentclass{siamart190516}

\usepackage{lipsum}
\usepackage{amsfonts}
\usepackage{graphicx}
\usepackage{epstopdf}
\usepackage{algorithmic}
\usepackage{subcaption}

\usepackage{lmodern}
\usepackage[T1]{fontenc}

\ifpdf
  \DeclareGraphicsExtensions{.eps,.pdf,.png,.jpg}
\else
  \DeclareGraphicsExtensions{.eps}
\fi

\newsiamremark{remark}{Remark}
\newsiamremark{hypothesis}{Hypothesis}
\crefname{hypothesis}{Hypothesis}{Hypotheses}
\newsiamthm{claim}{Claim}

\headers{Block Bidiagonalization for Matrix Approximation}{E. Hallman}

\title{A Block Bidiagonalization Method for Fixed-Accuracy Low-Rank Matrix Approximation\thanks{
This research was supported in part by the National Science Foundation through grant DMS-1745654.}}

\author{Eric Hallman\thanks{North Carolina State University
  (\email{erhallma@ncsu.edu}, \url{https://erhallma.math.ncsu.edu/}).} }

\usepackage{amsopn}

\definecolor{cerulean}{rgb}{0.16, 0.32, 0.75}

\newcommand{\R}{\mathbb{R}}

\newcommand{\mat}[1]{{\bf #1}}

\newcommand{\ts}{^T}
\newcommand{\eps}{\varepsilon}
\newcommand{\cmm}{C_{\text{mul}}}
\newcommand{\cqr}{C_{\text{qr}}}
\newcommand{\cqrcp}{C_{\text{qrcp}}}
\newcommand{\emach}{\epsilon_{\text{mach}}}

\DeclareMathOperator*{\Span}{Span}
\DeclareMathOperator*{\trace}{tr}

\ifpdf
\hypersetup{
  pdftitle={A Block Bidiagonalization Method for fixed-accuracy Low-Rank Matrix Approximation},
  pdfauthor={Eric Hallman}
}
\fi

\begin{document}

\maketitle

\begin{abstract}
	We present randUBV, a randomized algorithm for matrix sketching based on the block Lanzcos bidiagonalization process. Given a matrix $\mat{A}$, it produces a low-rank approximation of the form $\mat{UBV}\ts$, where $\mat{U}$ and $\mat{V}$ have orthonormal columns in exact arithmetic and $\mat{B}$ is block bidiagonal. In finite precision, the columns of both $\mat{U}$ and $\mat{V}$ will be close to orthonormal. Our algorithm is closely related to the randQB algorithms of Yu, Gu, and Li (2018) in that the entries of $\mat{B}$ are incrementally generated and the Frobenius norm approximation error may be efficiently estimated. It is therefore suitable for the fixed-accuracy problem, and so is designed to terminate as soon as a user input error tolerance is reached. Numerical experiments suggest that the block Lanczos method is generally competitive with or superior to algorithms that use power iteration, even when $\mat{A}$ has significant clusters of singular values. 
\end{abstract}



\section{Introduction}
In this paper we consider the problem of finding a quality low-rank approximation $\widetilde{\mat{A}}_r$ to a given matrix $\mat{A}\in \R^{m\times n}$, where we assume that $m\geq n$. In particular we consider the {\it fixed-accuracy} problem, where the desired truncation rank $r$ is not known in advance, but we instead want to find the smallest possible $r$ such that $\|\mat{A}-\widetilde{\mat{A}}_r\|_F < \tau$ for some tolerance $\tau$. 

The optimal approximation can be found by computing and truncating the SVD of $\mat{A}$, but when $\mat{A}$ is large this method may be impractically expensive. It is therefore increasingly common to use randomized techniques to find an approximation to the dominant subspace of $\mat{A}$: that is, to find a matrix $\mat{Q}\in \R^{m\times r}$ with orthonormal columns so that \cite{halko2011random}
\begin{equation}
	\mat{A}\approx \mat{QB}, 
\end{equation}
where $\mat{B}$ is an $r\times n$ matrix satisfying
\begin{equation}
	\mat{B} = \mat{Q}\ts \mat{A}. 
\end{equation}

Two variants on this basic approach are randomized subspace iteration and randomized block Lanczos. Algorithms \ref{alg:randqb} and \ref{alg:rbl} present prototype algorithms for each of these methods for the {\it fixed-rank} problem, where $r$ is specified in advance. 

\begin{algorithm}[H]
    \centering
    \caption{Randomized Subspace Iteration (\texttt{randQB}) \cite[Alg.~4.3]{halko2011random}}\label{alg:randqb}
    \begin{algorithmic}[1]
		\REQUIRE{$\mat{A}\in \R^{m\times n}$, rank $r$, integer $\ell\geq r$,  power parameter $p\geq 0$}
		\ENSURE{$\mat{Q}\in \R^{m\times \ell}$ with orthonormal columns, $\mat{B} \in \R^{\ell\times n}$}
        \STATE{Draw a random standard Gaussian matrix $\mat{\Omega}\in \R^{n\times \ell}$}
		\STATE{Form $\mat{Y} = (\mat{AA}\ts)^p\mat{A\Omega}$}
		\STATE{Compute the QR factorization $\mat{Y} = \mat{QR}$}
		\STATE{$\mat{B} = \mat{Q}\ts\mat{A}$}
    \end{algorithmic}
\end{algorithm}

\begin{algorithm}[H]
    \centering
    \caption{Randomized Block Lanczos \cite[Alg.~1]{yuan2018superlinear}}\label{alg:rbl}
    \begin{algorithmic}[1]
		\REQUIRE{$\mat{A}\in \R^{m\times n}$, block size $b\geq 1$, rank $r$, iterations $q$ such that $(q+1)b\geq r$}
		\ENSURE{$\mat{Q}\in \R^{m\times (q+1)b}$ with orthonormal columns, $\mat{B} \in \R^{(q+1)b\times n}$}
        \STATE{Draw a random standard Gaussian matrix $\mat{\Omega}\in \R^{n\times b}$}
		\STATE{Form $\mat{Y} = [\mat{A\Omega}, (\mat{AA}\ts)\mat{A\Omega}, \ldots,(\mat{AA}\ts)^{q}\mat{A\Omega}]$}
		\STATE{Compute the QR factorization $\mat{Y} = \mat{QR}$}
		\STATE{$\mat{B} = \mat{Q}\ts\mat{A}$}
    \end{algorithmic}
\end{algorithm}

Extensions of these algorithms to the fixed-accuracy problem make use of the fact that the columns of $\mat{Q}$ and rows of $\mat{B}$ can be computed incrementally rather than all at once. The process can then be terminated once a user-specified error threshold has been reached, assuming the error can be efficiently computed or estimated. Algorithms for the fixed-accuracy problem are proposed in \cite{halko2011random,martinsson2016randomized}, and more recently by Yu, Gu, and Li in \cite{yu2018efficient}. One algorithm by the latter authors, \texttt{randQB\_EI}, is currently the foundation for the MATLAB function \texttt{svdsketch} \cite{MATLAB:2020}. 

The algorithms cited above all rely on subspace iteration rather than the block Lanczos method, despite the fact that Krylov subspace methods are ``the classical prescription for obtaining a partial SVD'' \cite{halko2011random}, as with \texttt{svds} in MATLAB. One justification for the focus on subspace iteration is that convergence analysis is more complete. In particular, the block Lanczos method converges slowly when the spectrum of $\mat{A}$ has a cluster larger than the block size $b$, and the convergence analysis becomes more complicated in this situation. In recent years, however, several works have improved the analysis for randomized block Lanczos. Analyzing Algorithm \ref{alg:rbl} for the case $b\geq r$, Musco and Musco \cite{musco2015randomized} derive bounds on the approximation error that do not depend on the gaps between the singular values of $\mat{A}$. Yuan, Gu, and Li \cite{yuan2018superlinear} derive results under the more general condition where $\mat{A}$ has no singular values with multiplicity greater than $b$. Both papers focus mostly on theoretical results, but the latter authors make the following observation:

\begin{quote}
	``A practical implementation of [Algorithm \ref{alg:rbl}] should involve, at the very least, a reorganization of the computation to use the three-term recurrence and bidiagonalization \cite{golub1972lanczos}, and reorthogonalization of the Lanczos vectors at each step using one of the numerous schemes that have been proposed \cite{golub1972lanczos,parlett1979lanczos,simon1984lanczos}.''
\end{quote}

The goal of this paper is to provide a practical implementation of Algorithm \ref{alg:rbl}, along with a method for efficiently estimating the Frobenius norm approximation error. 

\subsection{Contributions}

Our main contribution is the algorithm \texttt{randUBV} (Algorithm \ref{alg:randubv}), which uses the block Lanczos method to solve the fixed accuracy problem. It is for the most part a straightforward combination of the block Lanzcos bidiagonalization process \cite{golub1981block} shown in Algorithm \ref{alg:bidiag} with a randomized starting matrix $\mat{V}_1 = \mat{\Omega}$. As such, it yields a factorization of the form $\mat{UBV}\ts$, where $\mat{U}$ and $\mat{V}$ have orthonormal columns in exact arithmetic and $\mat{B}$ is block bidiagonal. Our secondary contribution is Theorem \ref{thm:accuracy}, which establishes bounds on the accuracy of the Frobenius norm error estimate \eqref{eqn:errorEstimate}. 

Our algorithm has two notable features that make it competitive with methods based on subspace iteration:
\begin{itemize}
	\item It accepts block sizes smaller than the target rank. Contrary to what an exact arithmetic analysis would suggest, the block Lanczos method can find multiple singular values of $\mat{A}$ even when the multiplicity is greater than the block size $b$. Large clusters in the spectrum of $\mat{A}$ are inconvenient, but not fatal. 
	
	We can therefore compare \texttt{randUBV} with adaptive methods such as \texttt{randQB\_EI} when the two are run with the same block size. They will have the same cost per iteration when the latter algorithm is run with power parameter $p=0$, and empirically \texttt{randUBV} converges faster. If \texttt{randQB\_EI} instead uses $p=1$ or $p=2$ then \texttt{randUBV} empirically requires more iterations to converge, but each iteration costs significantly less. 
	\item It uses {\it one-sided reorthogonalization}, wherein $\mat{V}$ is reorthogonalized but $\mat{U}$ is not. This technique was recommended in \cite{simon2000low} for the single-vector case (i.e., $b=1$), and leads to considerable cost savings when $\mat{A}$ is sparse and $m\gg n$. If $m \ll n$, our algorithm should be run on $\mat{A}\ts$ instead. The matrix $\mat{U}$ may slowly lose orthogonality in practice, but Theorem \ref{thm:accuracy} shows that our error estimate \eqref{eqn:errorEstimate} will still remain accurate. 
	
	For simplicity, we use full reorthogonalization on $\mat{V}$ as opposed to more carefully targeted methods such as those discussed in \cite{parlett1979lanczos,simon1984lanczos}. 
\end{itemize}

One other design choice merits discussion: {\it deflation} occurs when the blocks produced by the block Lanczos method are nearly rank-deficient and results in a reduction in the block size. In the event of deflation, we propose to augment the block Krylov space in order to keep the block column size constant. This will prevent the process from terminating early in extreme cases such as when $\mat{A}$ is the identity matrix. 

Numerical experiments on synthetic and real data suggest that \texttt{randUBV} generally compares favorably with \texttt{randQB} and its variants, at least on modestly sized problems.

\subsection{Outline}
The paper is organized as follows. In section \ref{sec:background}, we review the background of \texttt{QB} algorithms for the fixed-accuracy problem as well as the block Lanczos method. In section \ref{sec:main} we discuss several implemenation details including the choice of block size, deflation and augmentation, and one-sided reorthogonalization. We present our main algorithm in section \ref{sec:algorithm} and establish the accuracy of the error indicator. Our numerical experiments are in section \ref{sec:experiments}, and section \ref{sec:conclusions} offers our concluding remarks and some avenues for future exploration. 

\subsection{Notation}
Matrices, vectors, integers, and scalars will be respectively denoted by $\mat{A}$, $\mat{a}$, $a$, and $\alpha$. We use $\|\mat{A}\|_F$ and $\|\mat{A}\|_2$ for the Frobenius norm and operator norm, respectively, and $\mat{I}$ for the identity matrix whose dimensions can be inferred from context. We use MATLAB notation for matrix indices: i.e., $\mat{A}(i,j)$ and $\mat{A}(:,j)$ respectively represent the $(i,j)$ element and the $j$-th column of $\mat{A}$. 

For the cost analysis of our algorithm we use the same notation as in \cite{martinsson2016randomized,yu2018efficient}: $\cmm$ and $\cqr$ will represent constants so that the cost of multiplying two dense matrices of sizes $m\times n$ and $n\times l$ is taken to be $\cmm mnl$ and the cost of computing the QR factorization of an $m\times n$ matrix with $m\geq n$ is taken to be $\cqr mn^2$, or $\cqrcp mn^2$ if column pivoting is used.

\section{Background} \label{sec:background}

In this section we review the fixed-accuracy QB factorization algorithm \texttt{randQB\_EI} and the block Lanczos bidiagonalization process.

\subsection{A fixed-accuracy QB algorithm}

In order to extend Algorithm \ref{alg:randqb} to the fixed-accuracy problem, Yu, Gu, and Li \cite{yu2018efficient} make use of two key ideas. First, for a given block size $b\leq \ell$ the matrix $\mat{\Omega}$ can be generated $b$ columns at a time rather than all at once, allowing the resulting factors $\mat{Q}$ and $\mat{B}$ to be generated incrementally. Second, since $\mat{Q}$ has orthonormal columns and $\mat{B} = \mat{Q}\ts \mat{A}$, it follows \cite[Thm.~1]{yu2018efficient} that 
\begin{equation}
	\|\mat{A}-\mat{QB}\|_F^2 = \|\mat{A} - \mat{QQ}\ts \mat{A}\|_F^2 = \|\mat{A}\|_F^2 - \|\mat{QQ}\ts \mat{A}\|_F^2 = \|\mat{A}\|_F^2 - \|\mat{B}\|_F^2. 
\end{equation}
As long as the columns of $\mat{Q}$ are kept close to orthonormal, the Frobenius norm error can be efficiently estimated at each step simply by updating $\|\mat{B}\|_F$. It is therefore possible to compute the low-rank factorization $\mat{QB}$ and cheaply estimate its error without ever forming the error matrix $\mat{A}-\mat{QB}$ explicitly. Algorithm \texttt{randQB\_EI} incorporates both of these ideas, the second of which is particularly useful when $\mat{A}$ is sparse. 

Algorithm \ref{randqb_ei} presents code for \texttt{randQB\_EI}, which in exact arithmetic will output the same $\mat{QB}$ factorization as \texttt{randQB} when run to the same rank. It is noted in \cite{halko2011random} that a stable implementation of Algorithm \ref{alg:randqb} should include a reorthogonalization step after each application of $\mat{A}$ or $\mat{A}\ts$. The reorthogonalization step in Line \ref{line:reorthq} provides further stability. 

\begin{algorithm}
    \centering
    \caption{Blocked randQB algorithm (\texttt{randQB\_EI}) \cite[Alg.~2]{yu2018efficient}}\label{randqb_ei} 
    \begin{algorithmic}[1]
		\REQUIRE{$\mat{A}\in \R^{m\times n}$, block size $b\geq 1$, power parameter $p\geq 0$, tolerance $\tau$}
		\ENSURE{$\mat{Q}\in \R^{m\times \ell}$, $\mat{B}\in \R^{\ell\times n}$, such that $\|\mat{A}-\mat{QB}\|_F< \tau$}
		\STATE{$\mat{Q} = [\ ]$, $\mat{B} = [\ ]$}
		\STATE{$E = \|\mat{A}\|_F^2$}\hfill \COMMENT{(Approximate costs)}
		\FOR{$k = 1,2,3,\ldots$} 
		\STATE{Draw a random standard Gaussian matrix $\mat{\Omega}_k\in\R^{n\times b}$}
		\STATE{$\mat{Q}_k = \text{qr}(\mat{A\Omega}_k - \mat{Q}(\mat{B\Omega}_k))$}\hfill\COMMENT{$\cmm mnb + (k-1)\cmm(m+n)b^2 + \cqr mb^2$ }\label{line:stabq}
		\FOR{$j = 1:p$}
		\STATE{$\tilde{\mat{Q}}_k = \text{qr}(\mat{A}\ts\mat{Q}_k - \mat{B}\ts(\mat{Q}\ts \mat{Q}_k))$}\hfill\COMMENT{$\text{\textemdash}''\text{\textemdash} + \text{\textemdash\textemdash\textemdash\textemdash\textemdash}''\text{\textemdash\textemdash\textemdash\textemdash\textemdash} + \cqr nb^2$ }
		\STATE{$\mat{Q}_k = \text{qr}(\mat{A}\tilde{\mat{Q}}_k - \mat{Q}(\mat{B}\tilde{\mat{Q}}_k))$}\hfill\COMMENT{$\text{\textemdash}''\text{\textemdash} + \text{\textemdash\textemdash\textemdash\textemdash\textemdash}''\text{\textemdash\textemdash\textemdash\textemdash\textemdash} + \cqr mb^2$ }
		\ENDFOR
		\STATE{$\mat{Q}_k = \text{qr}(\mat{Q}_k - \mat{Q}(\mat{Q}\ts \mat{Q}_k))$}\hfill\COMMENT{$2(k-1)\cmm mb^2 + \cqr mb^2$}\label{line:reorthq}
		\STATE{$\mat{B}_k = \mat{Q}_k\ts\mat{A}$}\hfill\COMMENT{$\cmm mnb$ }
		\STATE{$\mat{Q} = [\mat{Q},\, \mat{Q}_k]$}
		\STATE{$\mat{B} = \begin{bmatrix}
			\mat{B}\ts,\, \mat{B}_k\ts
		\end{bmatrix}\ts$}
		\STATE{$E = E - \|\mat{B}_k\|_F^2$}
		\STATE{{\bf if} $E < \tau^2$ {\bf then stop}}
		\ENDFOR
    \end{algorithmic}
\end{algorithm}

Suppose that we stop Algorithm \ref{randqb_ei} after $t$ iterations, and set $\ell = tb$. The runtime of \texttt{randQB\_EI} can then be approximated as 
\begin{align}\label{cost:randqb}
	\begin{split}
	T_{\texttt{randQB\_EI}} &\approx 2\cmm mn\ell + \frac{1}{2}\cmm (3m+n)\ell^2 + \frac{2}{t}\cqr m\ell^2\\
		&\ \  + p\left( 2\cmm mn\ell + \cmm(m+n)\ell^2 + \frac{1}{t}\cqr (m+n)\ell^2\right),
	\end{split}
\end{align}
where the cost increases more or less proportionally to $p+1$. By comparison, the cost of the fixed-rank prototype algorithm \texttt{randQB} can be approximated as 
\begin{equation}\label{cost:qb}
	T_{\texttt{randQB}} \approx 2(p+1)\cmm mn\ell + \cqr m\ell^2. 
\end{equation}

\subsection{Block Lanczos bidiagonalization}
Here we describe a block Lanczos method for reducing a matrix to block bidiagonal form. Since this method generalizes the single-vector algorithm by Golub and Kahan \cite{golub1965calculating} commonly known as the Golub-Kahan-Lanczos process, we will abbreviate it as \texttt{bGKL}. 

The \texttt{bGKL} process was introduced by Golub, Luk, and Overton \cite{golub1981block} to find the largest singular values and associated singular vectors of a large and sparse matrix. Since then, it has been applied to both least squares problems \cite{karimi2006block,toutounian2015block} and total least squares problems \cite{bjorck2008block,hnetynkova2015band} with multiple right-hand sides. 

The process takes a matrix $\mat{A}\in \R^{m\times n}$ and matrix $\mat{V}_1 \in \R^{n\times b}$ with orthonormal columns, and after $k$ steps produces the orthonormal bases $\mat{U}_{(k)} = [\mat{U}_1, \cdots, \mat{U}_k]$ and $\mat{V}_{(k+1)} = [\mat{V}_1,\cdots, \mat{V}_{k+1}]$ satisfying
\begin{align*}
	\Span\left\{\mat{U}_{(k)}\right\} &= \Span\left\{ \mat{A}\mat{V}_1, \mat{A}(\mat{A}\ts\mat{A})\mat{V}_1,\ldots,\mat{A}(\mat{A}\ts\mat{A})^{k-1}\mat{V}_1\right\}, \\
	\Span\left\{\mat{V}_{(k+1)}\right\} &= \Span\left\{ \mat{V}_1, (\mat{A}\ts\mat{A})\mat{V}_1,\ldots,(\mat{A}\ts\mat{A})^{k}\mat{V}_1\right\}. 
\end{align*}
Furthermore, it produces the $kb\times (k+1)b$ block bidiagonal matrix
\begin{equation}
	\mat{B}_{k} = \begin{bmatrix}
		\mat{R}_1 & \mat{L}_2 & &  & \\
		& \mat{R}_2 & \ddots & &\\
		& & \ddots & \mat{L}_k & \\
		& & & \mat{R}_k & \mat{L}_{k+1}
	\end{bmatrix}
\end{equation}
so that at each step of the process the relations 
\begin{equation}\label{eqn:gkrelations}
	\mat{AV}_{(k)} = \mat{U}_{(k)}\mat{B}_k(:,1:kb)
	\quad\text{and}\quad \mat{A}\ts\mat{U}_{(k)} = \mat{V}_{(k+1)}\mat{B}_k\ts
\end{equation}
are satisfied. Assuming no loss of rank, the blocks $\{\mat{R}_i\}_{i=1}^k$ or $\{\mat{L}_i\}_{i=1}^{k+1}$ are respectively $b\times b$ upper and lower triangular. 

\begin{algorithm}
\caption{Block Lanczos bidiagonalization process (\texttt{bGKL}) \cite{golub1981block}}
\label{alg:bidiag}
\begin{algorithmic}[1]
\REQUIRE{$\mat{A}\in \R^{m\times n}$, matrix $\mat{V}_1 \in \R^{n\times b}$ with orthonormal columns}
\STATE{$\mat{U}_0 = \mat{0}$; $\mat{L}_1 = \mat{0}$}\hfill \COMMENT{(Approximate costs)}
\FOR{$k=1,2,\ldots$}
\STATE{$\mat{U}_{k}\mat{R}_{k} = \text{qr}(\mat{A}\mat{V}_{k} - \mat{U}_{k-1}\mat{L}_{k})$}\hfill\COMMENT{$\cmm mnb + \frac{1}{2}\cmm mb^2 + \cqr mb^2$} \label{line:rec1}
\STATE{$\mat{V}_{k+1}\mat{L}_{k+1}\ts = \text{qr}(\mat{A}\ts \mat{U}_{k} - \mat{V}_{k}\mat{R}_{k}\ts)$}\hfill\COMMENT{$\cmm mnb + \frac{1}{2}\cmm nb^2 + \cqr nb^2$} \label{line:rec2}
\ENDFOR
\end{algorithmic}
\end{algorithm}

The basic outline of the process is given in Algorithm \ref{alg:bidiag}, where the costs assume no loss of rank in the blocks $\{\mat{R}_i\}_{i=1}^k$ or $\{\mat{L}_i\}_{i=1}^{k+1}$. We note that the original algorithm in \cite{golub1981block} is organized so that $\mat{B}_k$ is square at the end of each iteration. Our current presentation more directly mimics the $\mat{QB}$ factorization, since $\mat{U}_{(k)}\mat{B}_k\mat{V}_{(k+1)}\ts = \mat{U}_{(k)}\mat{U}_{(k)}\ts \mat{A}$ by the second relation in \eqref{eqn:gkrelations}. It follows that in exact arithmetic the identity 
\begin{equation}\label{eqn:errorEstimate}
	\|\mat{A} - \mat{U}_{(k)}\mat{B}_k\mat{V}_{(k+1)}\ts\|_F^2 =
	\|\mat{A}\|_F^2 - \|\mat{B}_k\|_F^2
\end{equation}
 will hold, and so the \texttt{bGKL} process can be readily adapted to find a fixed-accuracy approximation to $\mat{A}$. 
 
 Suppose that we stop the process after $t$ iterations and set $\ell = tb$. The runtime of the \texttt{bGKL} process can then be approximated as 
 \begin{equation}\label{cost:bgkl}
 	T_{\texttt{bGKL}} \approx 2\cmm mn\ell + \frac{1}{2t}\cmm (m+n)\ell^2 + \frac{1}{t}\cqr (m+n)\ell^2.
 \end{equation}
 At this point, it is not fair to compare this cost to the cost of \eqref{cost:randqb} because we have not yet accounted for the cost of reorthogonalization in \texttt{bGKL}, which is necessary for stability. Nonetheless, it suggests that we may be able to obtain an algorithm based on \texttt{bGKL} that costs no more per iteration than \texttt{randQB\_EI} with power parameter $p=0$.

\section{Implementation details}
\label{sec:main}
In this section we discuss how to handle several important issues in the implementation of our fixed-accuracy algorithm. The first concerns the difficulty the Lanczos method encounters when $\mat{A}$ has large singular value clusters. The second is the matter of ensuring that the columns of $\mat{U}_{(k)}$ and $\mat{V}_{(k)}$ remain close to orthonormal, and the third is the use of deflation and augmentation when the blocks $\mat{R}_k$ or $\mat{L}_{k}$ are rank-deficient.

\subsection{Block size and singular value clusters}
It is known that if $\mat{A}$ has a singular value with multiplicity greater than the block size $b$, then in exact arithmetic the block Lanczos process will recover at most $b$ of those singular values. More generally, if the spectrum of $\mat{A}$ has a cluster of size greater than $b$ then the approximate singular vectors recovered by the Lanczos process may converge slowly. This behavior stands in stark contrast to that of blocked subspace iteration methods such as \texttt{randQB\_EI}, whose outputs do not in exact arithmetic depend on $b$. 

For the first situation\textemdash singular values with multiplicity greater than $b$\textemdash classical results tend to examine a restricted problem. Saad \cite{saad1980rates} notes that the Lanczos process would simply behave as though it were being performed on a restricted matrix $\mat{A}|_{\mat{S}}$ whose singular values\footnote{Strictly speaking, Saad's analysis is for block Lanczos tridiagonalization applied to a symmetric matrix as opposed to Lanczos bidiagonalization applied a rectangular matrix. Our focus is on bidiagonalization, but the two processes are closely related. \label{note:lanczos}} had multiplicity at most $b$. There is therefore ``no loss of generality'' in assuming that the singular values of $\mat{A}$ have multiplicity bounded by $b$ for the purpose of analyzing convergence rates. Other more recent works restrict their attention to the case where the cluster size is bounded by $b$ \cite{li2015convergence}, or where $b$ is greater than or equal to the target rank $r$ \cite{musco2015randomized,wang2015improved,drineas2018structural}. 

The analysis of Yuan, Gu, and Li \cite{yuan2018superlinear} makes an important advancement by allowing for cluster sizes (though not multiplicity) greater than $b$, and showing that even within a large cluster the recovered singular values will converge superlinearly in the number of Lanczos iterations. Their numerical experiments on real-world data suggest that smaller block sizes generally lead to faster convergence with respect to the number of flops expended. 
 
As it turns out, even singular values with multiplicity greater than $b$ are not fatal to the Lanczos process. Parlett \cite{parlett1979lanczos} notes that since ``rounding errors introduce components in all directions'', even repeated singular vectors\footnote{See footnote \ref{note:lanczos}.} will eventually be found. Simon and Zha \cite{simon2000low} add the caveat that the singular vectors will not converge in consecutive order: the Lanczos process will likely find several smaller singular values of $\mat{A}$ before it finds copies of the larger repeated ones. What we should expect in practice is that a singular value of multiplicity greater than $b$ (or a cluster of comparable size) will delay convergence, but not prevent it entirely. 

Thus in spite of complications in the {\it analysis} of the block Lanczos method, using a smaller block size can be quite effective in practice. Even when $\mat{A}$ has clusters larger than the block size, we can obtain a good approximation simply by increasing the number of Lanczos iterations. Our numerical experiments support this notion: although we can construct synthetic examples for which \texttt{randUBV} is inferior to methods that use subspace iteration, our algorithm performs quite well on a real-world example with large clusters. 

\subsubsection{Adaptive block size}
An alternate method for dealing with clusters is offered in \cite{ye1996adaptive} and explored further in \cite{bai1999able, zhou2008block}: instead of keeping the block size constant, we may periodically augment the block Krylov space with new vectors in order to better approximate clusters. The rough idea would be to monitor the singular values of $\mat{B}_k$, and to increase the block size $b$ so that it remains larger than the largest cluster in $\mat{B}_k$. For the sake of keeping the implementation of our algorithm simple, we leave this extension for future exploration. 

\subsection{One-sided reorthogonalization}
In exact arithmetic, the matrices $\mat{U}_{(k)}$ and $\mat{V}_{(k)}$ will have orthonormal columns. In practice, they will quickly lose orthogonality due to roundoff error, and so we must take additional steps to mitigate this loss of orthogonality. 

For the single-vector case $b=1$, Simon and Zha \cite{simon2000low} observe that it may suffice to reorthogonalize only one of $\mat{U}_{(k)}$ or $\mat{V}_{(k)}$ in order to obtain a good low-rank approximation. They suggest that if the columns of $\mat{V}_{(k)}$ alone are kept close to orthonormal, then $\mat{U}_{(k)}\mat{B}_{k}\mat{V}_{(k+1)}\ts$ will remain a good approximation to $\mat{A}$ regardless of the orthogonality of $\mat{U}_{(k)}$. Separately, experiments by Fong and Saunders \cite{fong2011lsmr} in the context of least-squares problems suggest that keeping $\mat{V}_{(k)}$ orthonormal to machine precision $\emach$ might be enough to keep $\mat{U}_{(k)}$ orthonormal to at least $\mathcal{O}(\sqrt{\emach})$, at least until the least-squares solver reaches a relative backward error of $\sqrt{\emach}$. For the sake of computational efficiency, we therefore choose to explicitly reorthogonalize $\mat{V}_{(k)}$ but not $\mat{U}_{(k)}$ (assuming that $m\geq n$).

Reorthogonalization can take up a significant portion of the runtime of our algorithm, particularly if $\mat{A}$ is sparse. However, it is known for the Lanczos process that orthogonality is lost only in the direction of singular vectors that have already converged \cite{paige1971computation}. Thus in a high-quality implementation, it should be possible to save time by orthogonalizing each block $\mat{V}_{k}$ against a smaller carefully chosen set of vectors obtained from $\mat{V}_{(k-1)}$ (see \cite{parlett1979lanczos,grcar1982analyses,simon1984lanczos} for a few such proposals). In our implementation, we use full reorthogonalization for simplicity. We note that even if $\mat{A}$ is square, full reorthogonalization will cost no more than the equivalent step in \texttt{randQB\_EI} (line \ref{line:reorthq} of Algorithm \ref{randqb_ei}). 

\subsection{Deflation}
In practice, the block Lanczos process may yield blocks $\mat{R}_k$ or $\mat{L}_k$ that are rank-deficient or nearly so. Here and with other block Krylov methods, it is typical to reduce the block size $b$ in response so that $\mat{R}_k$ and $\mat{L}_k$ retain full row rank and column rank, respectively. This process is known as {\it deflation}. For more background, we refer the reader to the survey paper by Gutknecht \cite{gutknecht2006block} and the references therein. 

In the context of solving systems with multiple right-hand sides, Gutknecht stresses that deflation is highly desirable. Indeed, when solving a system such as $\mat{AX} = \mat{B}$, it is precisely the dimension reduction resulting from deflation that gives block methods an advantage over methods that solve each right hand side separately. In this context, deflation might occur if $\mat{B}$ is itself rank-deficient, or if $\mat{B}$ has some notable rank structure in relation to the matrix $\mat{A}$. When running block Lanczos with a randomly chosen starting matrix $\mat{V}_1$ (i.e., $\mat{V}_1 = \text{qr}(\mat{\Omega})$ and $\mat{\Omega}$ is a standard Gaussian matrix), we do not expect deflation to occur frequently since $\mat{\Omega}$ is not likely to have any notable structure with respect to $\mat{A}$. Nonetheless, a reliable implementation should be prepared for the possibility, and so we examine the details here. 

Bj\"{o}rck \cite{bjorck2008block} proposes computing the QR factorizations in lines \ref{line:rec1}--\ref{line:rec2} of Algorithm \ref{alg:bidiag} using Householder reflections without column pivoting. The resulting matrix $\mat{B}_k$ will be not just block bidiagonal, but a banded matrix whose effective bandwidth begins at $b$ and decreases with each deflation. Hn\v{e}tynkov\'{a} et al.~\cite{hnetynkova2015band} refer to $\mat{B}_k$ as a {\it $b$-wedge shaped matrix}. If the effective bandwidth decreases to zero, the bidiagonalization process will terminate.

\begin{algorithm}
    \centering
    \caption{Deflated QR (\texttt{deflQR})} \label{alg:deflqr}
    \begin{algorithmic}[1]
		\REQUIRE{$\mat{X}\in \R^{m\times n}$, deflation tolerance $\delta$}
		\ENSURE{$\mat{Q}\in \R^{m\times s}$ with orthonormal columns, $\mat{R}\in \R^{s\times n}$}, rank $s$
		\STATE{Compute the pivoted QR factorization $\mat{X}\mat{\Pi} = \widehat{\mat{Q}}\widehat{\mat{R}}$}
		\STATE{Find the largest $s$ such that $|\widehat{\mat{R}}(s,s)| \geq \delta$}
		\STATE{$\mat{R} = \widehat{\mat{R}}(1:s,:)\mat{\Pi}\ts$}
		\STATE{$\mat{Q} = \widehat{\mat{Q}}(:,1:s)$}
    \end{algorithmic}
\end{algorithm}

We propose to instead use QR with column pivoting, which is slower and less elegant but simpler to implement in terms of readily available subroutines. The procedure is outlined in Algorithm \ref{alg:deflqr}, where the deflation tolerance $\delta$ is presumably somewhat larger than $\emach\|\mat{A}\|_2$. Lines \ref{line:rec1}--\ref{line:rec2} of Algorithm \ref{alg:bidiag} would use this modified routine in place of unpivoted QR, and as Bj\"{o}rck \cite{bjorck2008block} notes the recurrence in those lines will still work in the presence of deflation. 
 
\subsection{Augmentation}
When using block Lanczos to solve systems of linear equations, deflation can be highly beneficial. In the context of matrix sketching, it is less desirable. Consider an extreme example where the columns of $\mat{V}_1$ are right singular vectors of $\mat{A}$: the Lanczos process will terminate after a single iteration, returning an approximation of the form $\mat{A} \approx \mat{U}_1\mat{\Sigma}\mat{V}_1\ts$. Termination at this point would yield accurate singular vectors, but the factorization may not approximate $\mat{A}$ to within the desired error tolerance. 

As mentioned before, we do not expect deflation to occur frequently if $\mat{V}_1$ is chosen randomly. However, if we do not make any further adjustments for deflation our algorithm would fail to converge on cases as simple as $\mat{A} = \mat{I}$. In order to make our method more robust, we will replace any deflated vectors with new randomly drawn ones in order to keep the block column size constant. Similar augmentation techniques have been proposed to prevent breakdown in the case of the nonsymmetric Lanczos process \cite{ye1994breakdown} and GMRES \cite{reichel2005breakdown}. 

More specifically, if Algorithm \ref{alg:deflqr} returns a factorization $\mat{V}_k\mat{L}_k\ts$ with rank less than $b$,  we generate a standard Gaussian matrix $\mat{\Omega}_k$ so that $[\mat{V}_k,\ \mat{\Omega}_k]$ has $b$ columns. We then orthogonalize $\mat{\Omega}_k$ against $\mat{V}_k$ and $\mat{V}_{(k-1)}$, obtaining $\mat{V}_k'$. The resulting matrix $[\mat{V}_k,\mat{V}_k']$ is then used in place of $\mat{V}_k$ in the next step of the Lanczos process. 

 In keeping with the spirit of one-sided reorthogonalization, we do not augment $\mat{U}_k$ if a block $\mat{R}_k$ is found to be rank deficient. This will allow us to avoid accessing the matrix $\mat{U}_{(k-1)}$ while the block Lanczos process is running. As a consequence, the blocks of $\mat{B}_k$ will each have $b$ columns, but some may have fewer than $b$ rows. 

We observe that in the presence of augmentation, the space $\text{Span}\left\{\mat{V}_{(k)}\right\}$ will not be a block Krylov space, but will instead be the sum of multiple block Krylov spaces with different dimensions. As of the time of writing we are not aware of any convergence results for this more general case.

\section{Fixed-accuracy algorithm} \label{sec:algorithm}
Algorithm \ref{alg:randubv} presents code for $\texttt{randUBV}$. Ignoring the augmentation step in line \ref{line:augment}, the cost is more or less equal to the cost of \texttt{bGKL} plus the cost of reorthogonalizing $\mat{V}_{k+1}$ in Line \ref{line:reorthV}. Thus if we stop the process after $t$ iterations and set $\ell = tb$, the total cost is approximately
\begin{equation}\label{cost:randubv}
	T_{\texttt{randUBV}} \approx 2\cmm mn\ell + \cmm n\ell^2+\frac{1}{2t}\cmm (m+n)\ell^2 + \frac{1}{t}\cqr(m+n)\ell^2. 
\end{equation}
Comparing this quantity to \eqref{cost:randqb}, we see that \texttt{randUBV} requires fewer floating points operations than \texttt{randQB\_EI} when run for the same number of iterations, even when the latter is run with power parameter $p=0$. In particular, the cost of one-sided reorthogonalization is only $\mathcal{O}(n\ell^2)$ while the stabilization steps in lines \ref{line:stabq} and \ref{line:reorthq} of \texttt{randQB\_EI} cost $\mathcal{O}((m+n)\ell^2)$. We can therefore expect that if $\mat{A}$ is sparse and $m\gg n$, \texttt{randUBV} may run significantly faster. 

\begin{algorithm}
    \centering
    \caption{Blocked Bidiagonalization algorithm (\texttt{randUBV})}\label{alg:randubv}
    \begin{algorithmic}[1]
		\REQUIRE{$\mat{A}\in \R^{m\times n}$, block size $b$, relative error $\tau$, deflation tolerance $\delta$}
		\ENSURE{$\mat{U}$, $\mat{B}$, $\mat{V}$, such that $\|\mat{A}-\mat{UBV}\ts\|_F< \tau$}
		\STATE{$E = \|\mat{A}\|_F^2$}\hfill\COMMENT{(Approximate costs)}
		\STATE{Draw a random standard Gaussian matrix $\mat{\Omega}\in \R^{n\times b}$}
		\STATE{$\mat{V}_1 = \text{qr}(\mat{\Omega})$}\hfill\COMMENT{$\cqr nb^2$}
		\STATE{$\mat{U}_1 = \mat{0}$; $\mat{L}_1 = \mat{0}$}
		\STATE{$\mat{V} = \mat{V}_1$; $\mat{U} = \mat{U}_1$}
		\FOR{$k = 1,2,3,\ldots$}
		\STATE{$[\mat{U}_k,\mat{R}_k] = \texttt{deflQR}(\mat{A}\mat{V}_{k}-\mat{U}_{k-1}\mat{L}_k,\delta)$}\hfill\COMMENT{$\cmm mnb + \frac{1}{2}\cmm mb^2 + \cqrcp mb^2$}
		\STATE{$\mat{U} = [\mat{U},\mat{U}_k]$}
		\STATE{$E = E - \|\mat{R}_k\|_F^2$}
		\STATE{$\mat{V}_{k+1} = \mat{A}\ts \mat{U}_{k} - \mat{V}_{k}\mat{R}_{k}\ts$} \hfill\COMMENT{$\cmm mnb + \frac{1}{2}\cmm nb^2$}
		\STATE{$\mat{V}_{k+1} = \mat{V}_{k+1} - \mat{V}(\mat{V}\ts \mat{V}_{k+1})$}\hfill\COMMENT{$2k\cmm nb^2$} \label{line:reorthV}
		\STATE{$[\mat{V}_{k+1},\mat{L}_{k+1}\ts,s] = \texttt{deflQR}(\mat{V}_{k+1},\delta)$}\hfill\COMMENT{$\cqrcp nb^2$}
		\STATE{$\mat{V} = [\mat{V},\mat{V}_{k+1}]$}
		\IF{$s < b$}
		\STATE{Draw a random standard Gaussian matrix $\mat{\Omega}_k\in \R^{n\times (b-s)}$}
		\STATE{$\mat{V}_{k+1}' = \text{qr}(\mat{\Omega}_k - \mat{V}(\mat{V}\ts \mat{\Omega}_k))$}\hfill\COMMENT{$2k\cmm nb(b-s) + \cqr n(b-s)^2$}\label{line:augment}
		\STATE{$\mat{V} = [\mat{V},\mat{V}_{k+1}']$}
		\ENDIF
		\STATE{$E = E - \|\mat{L}_{k+1}\|_F^2$}
		\STATE{{\bf if} $E < \tau^2\|\mat{A}\|_F^2$ {\bf then stop}}
		\ENDFOR
    \end{algorithmic}
\end{algorithm}

Since our focus is on the fixed-accuracy algorithm, however, different algorithms (and for \texttt{randQB\_EI}, different power parameters $p$) will converge after different numbers of iterations. We must therefore consider not just the cost per iteration, but how quickly the approximations converge. We discuss this matter further along with the numerical experiments in section \ref{sec:experiments}.

\subsection{Approximation accuracy}
It is noted in \cite{yu2018efficient} that due to cancellation, the computed value of $E = \|\mat{A}\|_F^2 - \|\mat{B}\|_F^2$ may be inaccurate when $E$ is very small. In order to estimate the error $E$ to within a relative tolerance of $\gamma$ (say, $\gamma = 1\%$), the authors suggest that the absolute accuracy tolerance $\tau$ for the QB factorization should be set large enough to satisfy
\begin{equation}
	\tau > \sqrt{E} \geq \sqrt{\frac{4\epsilon_\text{mach}}\gamma}\|\mat{A}\|_F,
\end{equation}
where $\epsilon_{\text{mach}}$ is the machine precision. In short, the proposed method of error estimation cannot reliably estimate a relative error below $2\sqrt{\epsilon_\text{mach}}$. 

We provide a similar analysis in order to account for deflation and loss of orthogonality of $\mat{U}_{(k)}$. In particular, we show that the error estimate can remain accurate even as $\mat{U}_{(k)}$ loses orthogonality in practice. To that end, we define the {\it local loss of orthogonality} of a matrix as follows: 
\begin{definition}\label{def:epsk}
	Given a matrix $\mat{U}_{(k)} = [\mat{U}_1,\ldots,\mat{U}_k]$, the local loss of orthogonality of $\mat{U}_{(k)}$ is defined as  
	\[
		\eps_k = \max \left\{\max_{1\leq i \leq k} \|\mat{U}_i\ts\mat{U}_i - \mat{I}\|_2,\ \max_{2\leq i\leq k} \|\mat{U}_{i-1}\ts\mat{U}_i\|_2 \right\} 
	\]
\end{definition}
The main idea is that we do not require $\|\mat{U}_{(k)}\ts \mat{U}_{(k)}-\mat{I}\|_2$ to be small. Instead, we need only the milder condition that adjacent blocks be close to orthogonal. This idea bears some resemblance to the work \cite{strakovs2002error}, which uses local recurrence formulas to show that certain error estimates for the conjugate gradient method remain accurate in a finite precision setting.

\begin{lemma}\label{lemma:local}
	Consider the matrix $\mat{U}_{(k)} = [\mat{U}_1,\ldots,\mat{U}_k]$, and let $\eps_k$ denote the local loss of orthogonality of $\mat{U}_{(k)}$. Let $\mat{B}_k$ be a block upper bidiagonal matrix whose blocks are partitioned conformally with those of $\mat{U}_{(k)}$. Then 
	\[\|\mat{U}_{(k)}\mat{B}_k\|_F^2 = (1+\theta)\|\mat{B}_k\|_F^2, \quad |\theta| \leq 2\eps_k.\]
\end{lemma}

\begin{proof}
	We will find the squared Frobenius norm of $\mat{U}_{(k)}\mat{B}_k$ one block column at a time, and use the fact that since $\mat{B}_k$ is block bidiagonal, each block column in the product uses at most two adjacent blocks of $\mat{U}_{(k)}$. 
	
	Let $\{\mat{R}_i\}_{i=1}^k$ denote the blocks on the main block diagonal of $\mat{B}_k$, and let $\{\mat{L}_i\}_{i=2}^{k+1}$ denote the off-diagonal blocks. Then for $2\leq i \leq k$, the squared Frobenius norm of the $i$-th block column of $\mat{U}_{(k)}\mat{B}_k$ is given by 
	\begin{equation}
				\|\mat{U}_{i-1}\mat{L}_{i} + \mat{U}_i\mat{R}_i\|_F^2 = \|\mat{U}_{i-1}\mat{L}_i\|_F^2 + \|\mat{U}_i\mat{R}_i\|_F^2 + 2\trace\left( \mat{R}_i\ts \mat{U}_i\ts \mat{U}_{i-1}\mat{L}_i\right).
	\end{equation}
	Examining the first term, it can be seen that
	\begin{align*}
		\|\mat{U}_{i-1}\mat{L}_i\|_F^2 &= \trace( \mat{L}_i\ts \mat{U}_{i-1}\ts\mat{U}_{i-1}\mat{L}_i)\\
		&= \trace( \mat{L}_i\ts( \mat{U}_{i-1}\ts\mat{U}_{i-1}-\mat{I})\mat{L}_i) + \trace( \mat{L}_i\ts\mat{L}_i)\\
		&= (1+\theta_1)\|\mat{L}_i\|_F^2,
	\end{align*}
	where $|\theta_1|\leq \eps_k$. A similar result applies to the term $\|\mat{U}_i\mat{R}_i\|_F^2$. As for the final term, we find that  
	\begin{align*}
		2|\trace \mat{R}_i\ts \mat{U}_i\ts \mat{U}_{i-1}\mat{L}_i| &\leq 2\| \mat{U}_i\ts \mat{U}_{i-1}\|_2\|\mat{R}_i\|_F\|\mat{L}_i\|_F\\
		&\leq 2\eps_k\|\mat{R}_i\|_F\|\mat{L}_i\|_F, \\
		&\leq \eps_k(\|\mat{R}_i\|_F^2 + \|\mat{L}_i\|_F^2).
	\end{align*}
	By adding these expressions back together we arrive at the bound
	\begin{equation}
				\|\mat{U}_{i-1}\mat{L}_{i} + \mat{U}_i\mat{R}_i\|_F^2 = (1+\theta)(\|\mat{R}_i\|_F^2 + \|\mat{L}_i\|_F^2), \quad |\theta|\leq 2\eps_k,
	\end{equation}
	so the desired relative error bound holds for each block column (the first and last columns may be checked separately). The main claim then follows by summing over the block columns. 
\end{proof}

Next, we observe that with one-sided reorthogonalization of $\mat{V}_{(k)}$ and in the absence of deflation, the {\it first} relation in \eqref{eqn:gkrelations} will remain accurate to machine precision regardless of the orthogonality of $\mat{U}_{(k)}$ (as noted in \cite{simon2000low}, the second relation will not). In the presence of deflation, the first relation must be amended slightly. We rewrite it as
\begin{equation}\label{eqn:newgkrelation}
	\mat{A}\mat{V}_{(k)} = \mat{U}_{(k)}\mat{B}_k' + \mat{D}_{k},
\end{equation}
where $\mat{B}_k'$ is shorthand for $\mat{B}_k(:,1:kb)$ and $\mat{D}_k$ is a matrix accounting for all deflations in $\mat{U}_{(k)}$. Assuming the column pivoting in Algorithm \ref{alg:deflqr} selects at each step the column with the largest 2-norm, it can be verified that $\|\mat{D}_k\|_F \leq \delta \sqrt{d}$, where $\delta$ is the deflation tolerance and $d$ is the total number of columns that have been removed from $\mat{U}_{(k)}$ through deflation. 

We now show that the error estimate $E = \|\mat{A}\|_F^2 - \|\mat{B}_k\|_F^2$ will remain accurate up to terms involving the deflation tolerance and the local loss of orthogonality in $\mat{U}_{(k)}$. The proof makes the simplifying assumptions that $\mat{V}_{(k+1)}$ has orthonormal columns and that there is no rounding error term in \eqref{eqn:newgkrelation}, but accounting for both of these effects will change the bound \eqref{eqn:accuracyBound} by at most $\mathcal{O}(\emach \|\mat{A}\|_F^2)$. The proof also ignores the effect of cancellation in the computation of $E$, so as with \cite{yu2018efficient} we cannot expect to reliably estimate a relative error below $\sqrt{\emach}$. 
\begin{theorem}\label{thm:accuracy}
	Given a matrix $\mat{A}$, let $\mat{U}_{(k+1)}$, $\mat{B}_{k+1}'$, and $\mat{V}_{(k+1)}$ be as produced by Algorithm \ref{alg:randubv} with deflation tolerance $\delta$. Let $\eps_{k+1}$ denote the local loss of orthogonality of $\mat{U}_{(k+1)}$. Assume that $\mat{V}_{(k+1)}$ has orthonormal columns. Assume that \eqref{eqn:newgkrelation} holds exactly at each iteration, and let $d$ be the number of columns removed from $\mat{U}_{(k+1)}$ due to deflation. Finally, let $E = \|\mat{A}\|_F^2 - \|\mat{B}\|_F^2$. Then 
	\begin{equation}\label{eqn:accuracyBound}
		\|\mat{A}-\mat{U}_{(k)}\mat{B}_k\mat{V}_{(k+1)}\ts \|_F^2 \leq E + 4\eps_{k+1}\|\mat{A}\|_F^2 + 2\delta\sqrt{d}(1+2\eps_{k+1})\|\mat{A}\|_F.
	\end{equation}
\end{theorem}
\begin{proof}
	 First, by assuming the columns of $\mat{V}_{(k+1)}$ are orthonormal we find that 
	\begin{equation} \label{eqn:errBound1}
		\|\mat{A}-\mat{U}_{(k)}\mat{B}_k\mat{V}_{(k+1)}\ts \|_F^2 
		=
		\|\mat{A}\|_F^2 + \|\mat{U}_{(k)}\mat{B}_k\|_F^2 - 2\trace (\mat{A}\mat{V}_{(k+1)}\mat{B}_k\ts \mat{U}_{(k)}\ts).
	\end{equation}
	By assuming that \eqref{eqn:newgkrelation} holds exactly at each step, we also get the identity
	\[
		\mat{A}\mat{V}_{(k+1)} = \mat{U}_{(k+1)}\mat{B}_{k+1}' + \mat{D}_{k+1} = \mat{U}_{(k)}\mat{B}_{k} + [\mat{0},\mat{U}_{k+1}\mat{R}_{k+1}] + \mat{D}_{k+1},
	\]
	where $\|\mat{D}_{k+1}\|_F \leq \delta \sqrt{d}$. It follows that
	\begin{equation}\label{eqn:errPart1}
		\trace (\mat{A}\mat{V}_{(k+1)}\mat{B}_k\ts \mat{U}_{(k)}\ts)
		=
		\|\mat{U}_{(k)}\mat{B}_k\|_F^2 + \trace(\mat{U}_k\ts \mat{U}_{k+1}\mat{R}_{k+1}\mat{L}_{k+1}\ts) + \trace(\mat{D}_{k+1}\mat{B}_k\ts \mat{U}_{(k)}\ts).
	\end{equation}
	From the definition of $\eps_{k+1}$ we have
	\begin{equation} \label{eqn:errPart2}
				\left|\trace(\mat{U}_k\ts \mat{U}_{k+1}\mat{R}_{k+1}\mat{L}_{k+1}\ts)\right| \leq \|\mat{U}_k\ts \mat{U}_{k+1}\|_2\|\mat{R}_{k+1}\|_F\|\mat{L}_{k+1}\|_F \leq \eps_{k+1}\|\mat{A}\|_F^2, 
	\end{equation}
	and since $\|\mat{D}_{k+1}\|_F \leq \delta\sqrt{d}$ we also have
	\begin{equation} \label{eqn:errPart3}
		\left|\trace (\mat{D}_{k+1} \mat{B}_k\ts \mat{U}_{(k)}\ts)\right| \leq \|\mat{D}_{k+1}\|_F\|\mat{U}_{(k)}\mat{B}_k\|_F \leq \delta\sqrt{d}\|\mat{U}_{(k)}\mat{B}_k\|_F.	
	\end{equation}

	Lemma \ref{lemma:local} gives us bounds on $\|\mat{U}_{(k)}\mat{B}_k\|_F^2$, so by returning to \eqref{eqn:errBound1} and using \eqref{eqn:errPart1}, \eqref{eqn:errPart2}, and \eqref{eqn:errPart3}, we conclude that 
	\begin{align*}
		\|\mat{A}-\mat{U}_{(k)}\mat{B}_k\mat{V}_{(k+1)}\ts \|_F^2 
		&=
		\|\mat{A}\|_F^2 + \|\mat{U}_{(k)}\mat{B}_k\|_F^2 - 2\trace (\mat{A}\mat{V}_{(k+1)}\mat{B}_k\ts \mat{U}_{(k)}\ts)\\
		&\leq \|\mat{A}\|_F^2 - \|\mat{U}_{(k)}\mat{B}_k\|_F^2 + 2\eps_{k+1}\|\mat{A}\|_F^2 + 2\delta\sqrt{d}\|\mat{U}_{(k)}\mat{B}_k\|_F\\
		&\leq E + 4\eps_{k+1}\|\mat{A}\|_F^2 + 2\delta\sqrt{d}(1+2\eps_{k+1})\|\mat{A}\|_F. 
	\end{align*}

\end{proof}

Thus as long as {\it local} orthogonality is maintained for $\mat{U}_{(k)}$ and as long as the number of deflations is not too large, we can expect $E$ to remain an accurate estimate of the Frobenius norm approximation error, at least when the error tolerance is not too small.

\subsection{Postprocessing of $\mat{B}$}
Recall that our original goal for the fixed-accuracy problem was not just to find a factorization that satisfies the bound $\|\mat{A}-\mat{UBV}\ts\|_F < \tau$, but to find the factorization with the smallest rank that does so. In order to accomplish this, we may compute the SVD of $\mat{B}$ as $\mat{B} = \hat{\mat{U}}\mat{\Sigma}\hat{\mat{V}}\ts$, truncate it to the smallest rank $r$ such that $\|\mat{A}-\hat{\mat{U}}_r\mat{\Sigma}_r\hat{\mat{V}}_r\ts\|_F < \tau$, then approximate the left and right singular vectors of $\mat{A}$ by $\mat{U}\hat{\mat{U}}_r$ and $\mat{V}\hat{\mat{V}}_r$. It should be noted that since $\mat{B}$ is a block bidiagonal matrix, its SVD can in theory be computed more efficiently than if $\mat{B}$ were dense. Algorithms for computing the SVD typically first reduce the matrix to bidiagonal form \cite{golub1965calculating}, and $\mat{B}$ can be efficiently reduced to this form using band reduction techniques as in \cite{kaufman2000band}. 

This postprocessing step takes on additional importance when dealing with the block Lanczos method rather than subspace iteration. Where subspace iteration will yield a matrix $\mat{B}$ whose singular values are all decent approximations of the top singular values of $\mat{A}$, the factor $\mat{B}$ produced by the Lanczos method will contain approximations to the {\it smallest} singular values of $\mat{A}$ as well \cite{golub2013matrix}. It is therefore possible that the matrix $\mat{B}$ produced by \texttt{randUBV} can be truncated significantly without diminishing the quality of the approximation. 

In fact, if one has the goal of obtaining a factorization whose rank is as small as possible, we recommend setting the stopping tolerance $\tau_{\text{stop}}$ slightly smaller than the desired approximation tolerance $\tau_{\text{err}}$ (or similarly, running the algorithm for a few more iterations after the approximation tolerance has already been satisfied). Doing so will may significantly reduce the rank $r$ of the truncated SVD, which will in turn pay dividends by reducing the cost of computing $\mat{U}\hat{\mat{U}}_r$ and $\mat{V}\hat{\mat{V}}_r$.

\section{Numerical experiments}\label{sec:experiments}
Here we report the results of numerical experiments on synthetic and real test cases. We run four sets of experiments in order to examine the following: 
\begin{enumerate}
	\item The rate of convergence by iteration. We use synthetic matrices whose spectra decay at different rates, and compare \texttt{randUBV} with \texttt{randQB\_EI} using power iterations $p=0,1,2$ for the latter. 
	\item The effect of sparsity and truncation rank on reorthogonalization costs. 
	\item The effect of block size on the time and number of iterations required for convergence. 
	\item The effect of choosing a smaller stopping tolerance $\tau_{\text{stop}}<\tau_{\text{err}}$ on the quality of the approximation. 
\end{enumerate}
All experiments were carried out in MATLAB 2020b on a 4-core Intel Core 7 with 32GB RAM. 

\begin{figure}[H]
	\begin{subfigure}{\textwidth}
		\centering
		\begin{minipage}{0.45\textwidth}
			\centering
			\includegraphics[width=\textwidth]{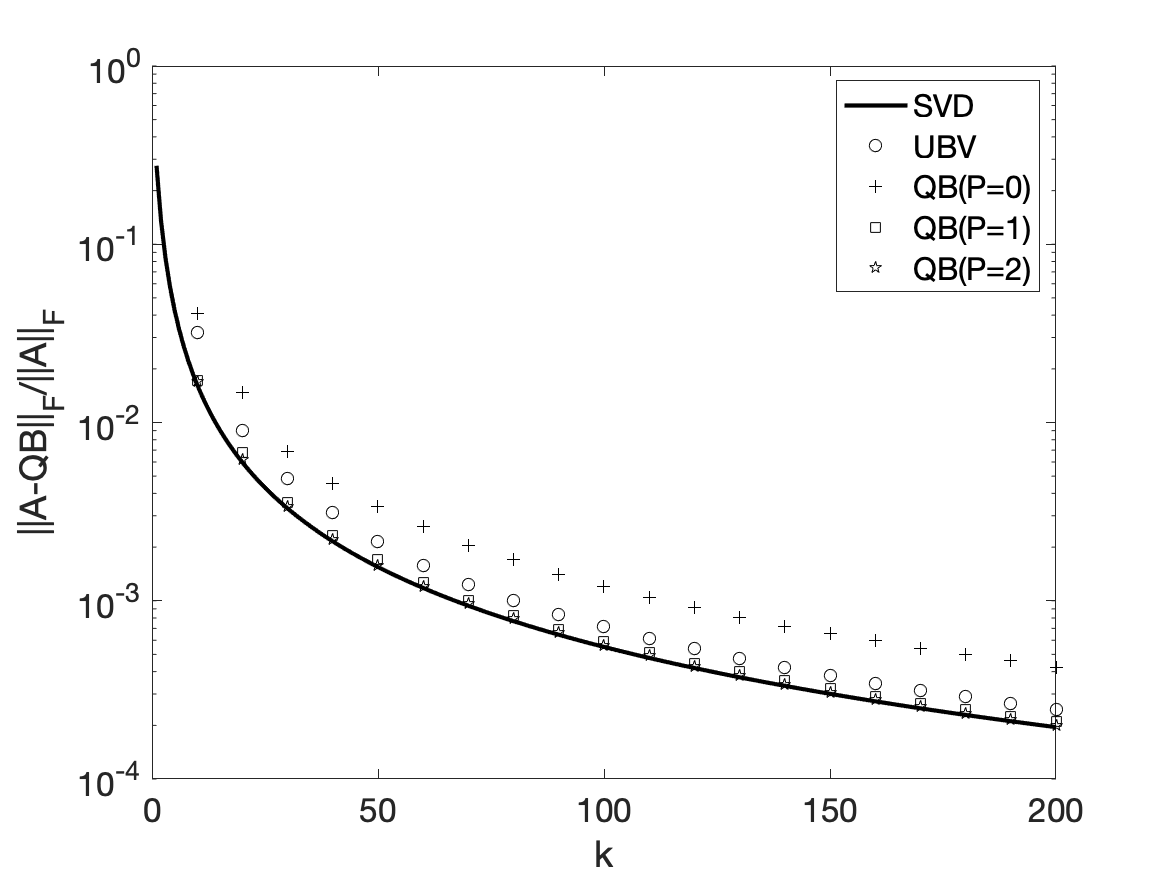}
		\end{minipage}
		\begin{minipage}{0.45\textwidth}
			\centering
			\includegraphics[width=\textwidth]{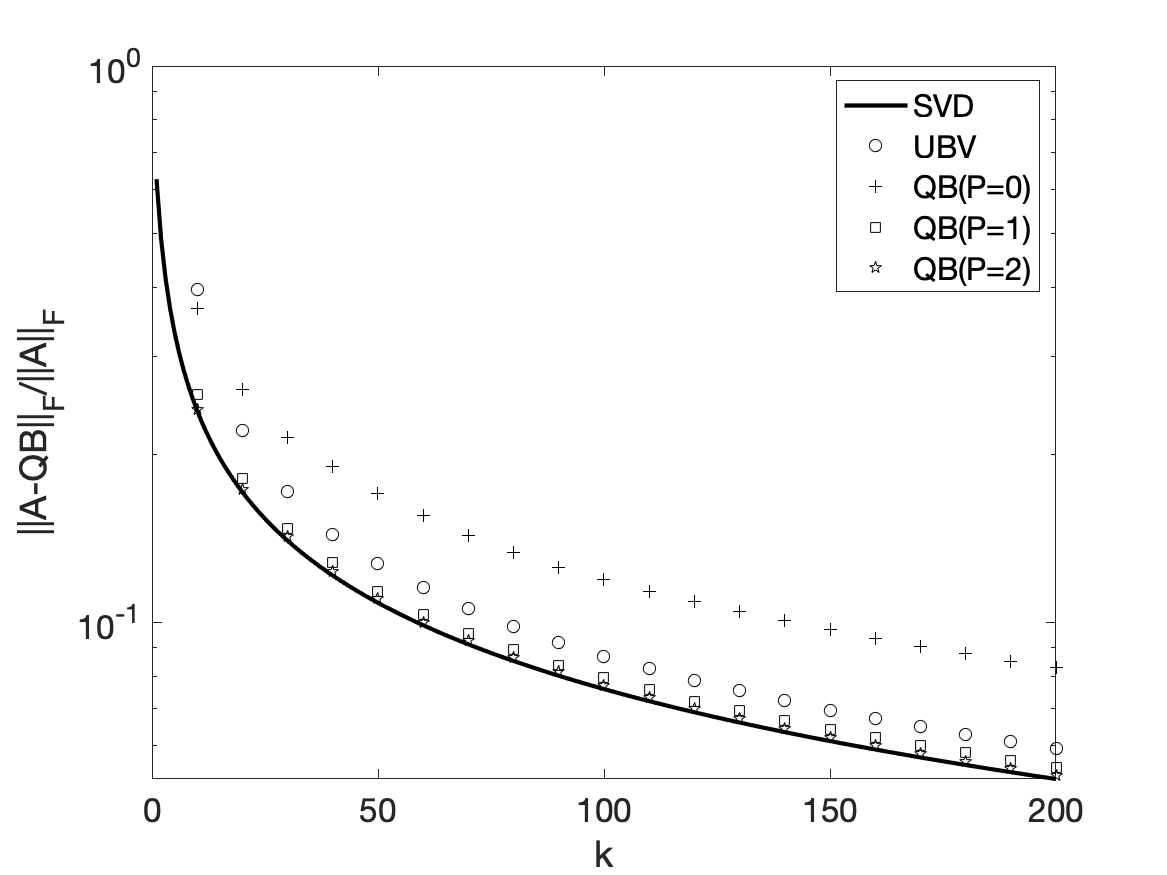}
		\end{minipage}
		\caption{Left: slow decay. Right: very slow decay. }
	\end{subfigure}
	\begin{subfigure}{\textwidth}
		\centering
		\begin{minipage}{0.45\textwidth}
			\centering
			\includegraphics[width=\textwidth]{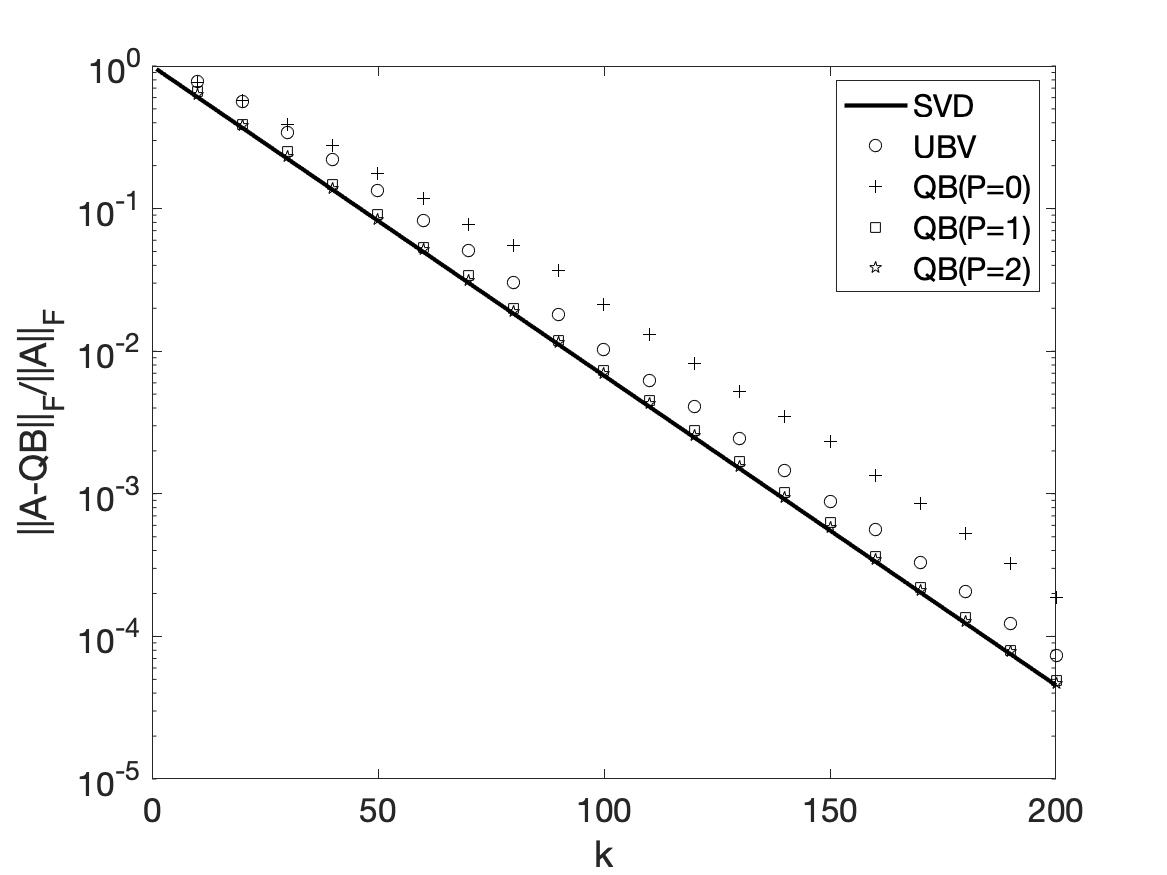}
		\end{minipage}
		\begin{minipage}{0.45\textwidth}
			\centering
			\includegraphics[width=\textwidth]{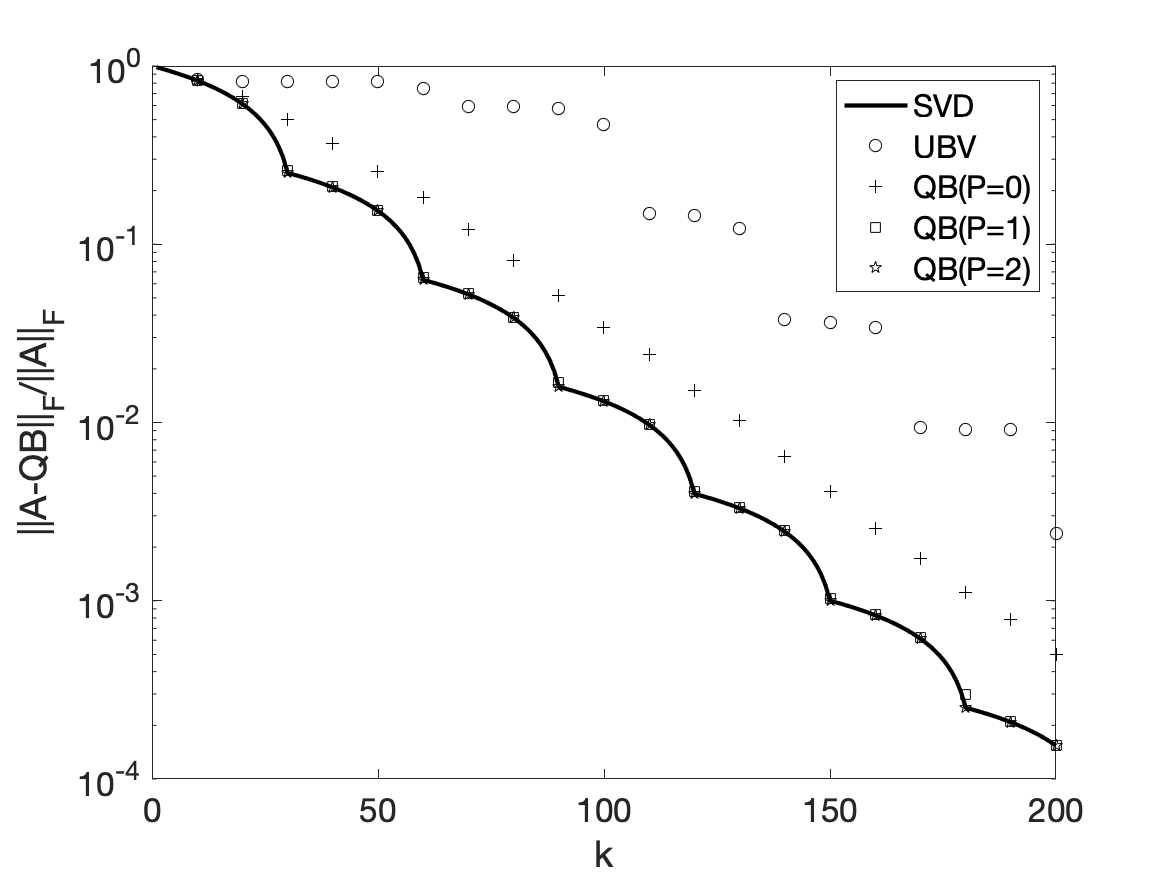}
		\end{minipage}
		\caption{Left: fast decay. Right: singular values have multiplicity greater than the block size.  }
	\end{subfigure}
	\caption{Convergence rate by iteration. In all cases but the last, \texttt{randUBV} requires fewer iterations for convergence than \texttt{randQB\_EI} with $p=0$ but more than \texttt{randQB\_EI} with $p=1$.}
	\label{fig:accTest}
\end{figure}

\subsection{Convergence rate by iteration}
For our first set of test cases we created matrices of size $2000\times 2000$ with the form $\mat{A} = \mat{U\Sigma V}\ts$, where $\mat{U}$ and $\mat{V}$ were formed by orthogonalizing standard Gaussian matrices and $\mat{\Sigma}$ was set in the following manner: 
\begin{itemize}
	\item (Matrix 1) Slow decay, in which $\sigma_j = 1/j^2$ for $1\leq j \leq 2000$. 
	\item (Matrix 2) Very slow decay, in which $\sigma_j = 1/j$ for $1\leq j \leq 2000$. 
	\item (Matrix 3) Fast decay, in which $\sigma_j = \exp(-j/20)$ for $1\leq j \leq 2000$. 
	\item (Matrix 4) Step function decay, in which $\sigma_j = 10^{-0.6(\lceil j/30\rceil - 1)}$ for $1\leq j \leq 2000$. Each singular value of $\mat{A}$ (except for the smallest) has multiplicity 30. 
\end{itemize}
In all four cases, we ran the sketching algorithms to a maximum rank $k= 200$ using block size $b=10$. The deflation tolerance was set at $\delta = 10^{-12}\sqrt{\|\mat{A}\|_1\|\mat{A}\|_\infty}$, but we did not encounter deflation in any of these cases.

Results are shown in Figure \ref{fig:accTest}. In the first three test cases, the approximation error for \texttt{randUBV} was smaller than that of \texttt{randQB\_EI} (with power parameter $p=0$) for every iteration after the first. It lagged somewhat behind \texttt{randQB\_EI} with $p=1$ or $p=2$, both of which were quite close to optimal. In the final case, where the singular values of $\mat{A}$ were chosen to have multiplicity larger than the block size, \texttt{randUBV} lagged significantly behind even \texttt{randQB\_EI} with $p=0$. We note that algorithm \texttt{randUBV} did nonethless converge, which would not have been possible in exact arithmetic. 

Finally, we offer a snapshot of the singular values of $\mat{B}_{200}$ after the algorithms have terminated. Results for test cases 1 and 4 are shown in Figure \ref{fig:svalB}. We note that the leading singular values returned by \texttt{randUBV} are more accurate than those returned by \texttt{randQB\_EI} with $p=0$ and comparable to the cases $p=1$ or $p=2$. The smallest singular values for \texttt{randUBV} are much smaller than their \texttt{randQB} counterparts, which appears to be undesirable but has a bit of a silver lining: it suggests that the rank of $\mat{B}_k$ can be truncated without losing much approximation accuracy.

\begin{figure}[H]
	\centering
	\begin{minipage}{0.45\textwidth}
		\includegraphics[width=\textwidth]{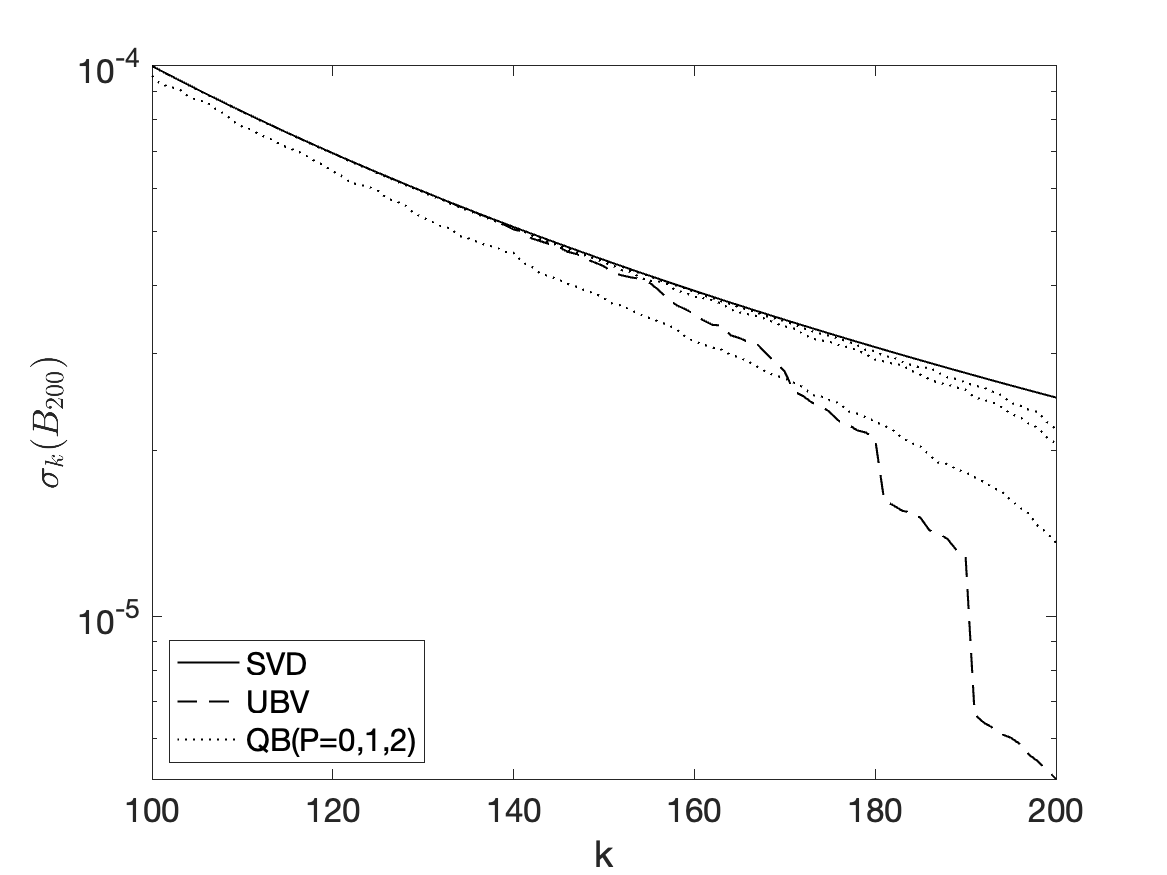}
	\end{minipage}
	\begin{minipage}{0.45\textwidth}
		\includegraphics[width=\textwidth]{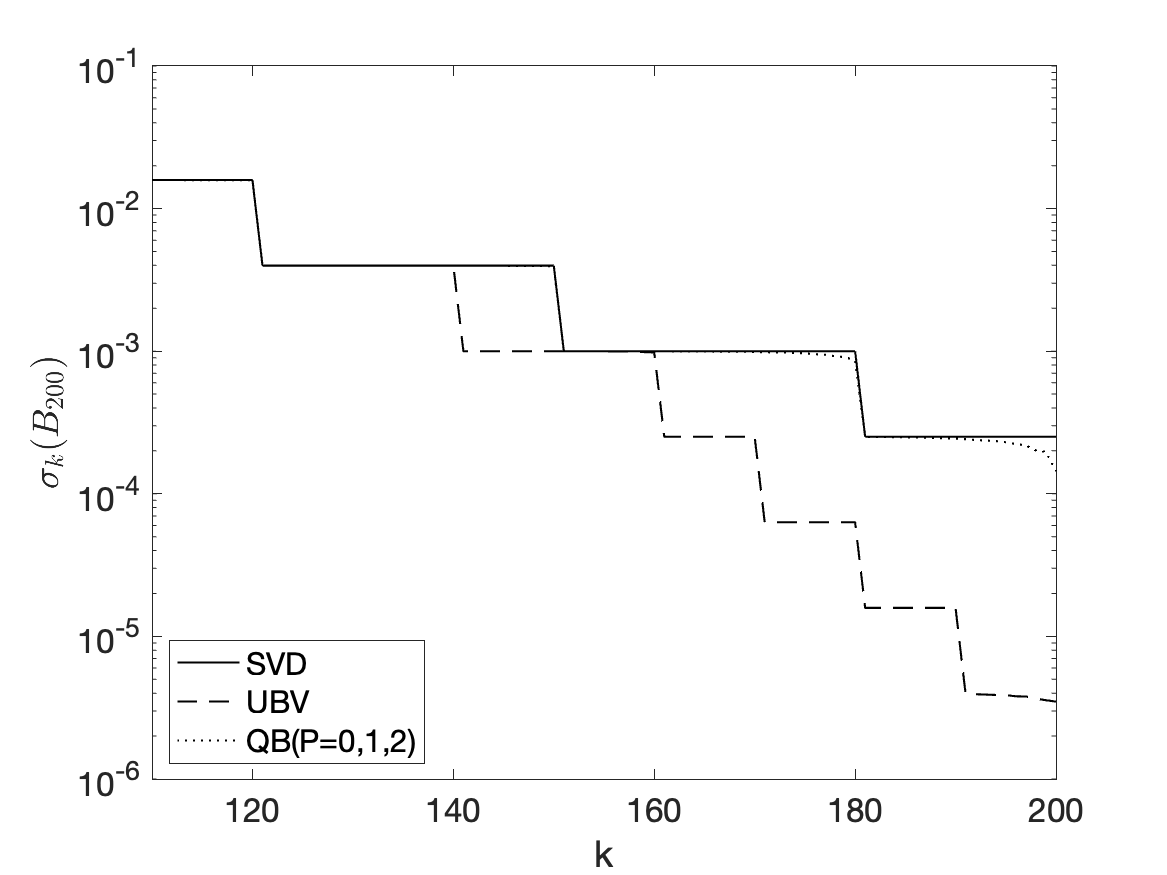}
	\end{minipage}
	\caption{Singular values of $\mat{B}_k$ after termination. Left: slow decay. Right: step function decay.}
	\label{fig:svalB}
\end{figure}

\subsection{Reorthogonalization costs}

For our second set of test cases we generated random sparse matrices as \texttt{A = sprand(m,n,d)} with $n=4000$ columns and varying numbers of rows $m$ and densities $d$. We then approximated \texttt{A} to a variable rank $k$ using \texttt{randUBV} and \texttt{randQB\_EI} with $p=0$. We tested three different variations: 
\begin{itemize}
	\item Number of rows $m$ varying from $8000$ to $40000$, rank $k=600$, and $d = 0.8\%$ nonzeros. 
	\item Number of rows $m=24000$, rank $k$ varying from $200$ to $1000$, and $d = 0.8\%$ nonzeros. 
	\item Number of rows $m=24000$, rank $k=600$, and nonzeros varying from $d=0.4\%$ to $d=2\%$.
\end{itemize}

\begin{figure}[H]
	\centering
	\begin{minipage}{0.45\textwidth}
		\includegraphics[width=\textwidth]{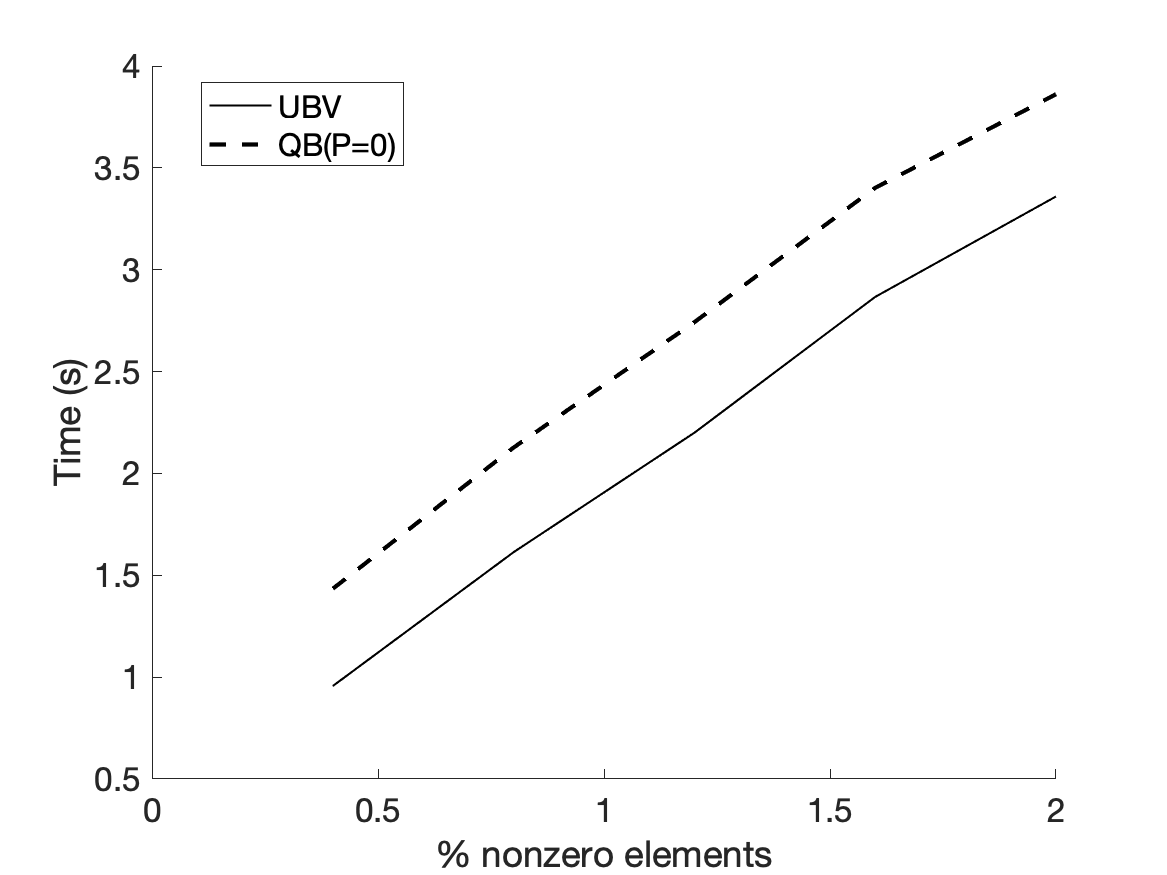}
	\end{minipage}
	\begin{minipage}{0.45\textwidth}
		\includegraphics[width=\textwidth]{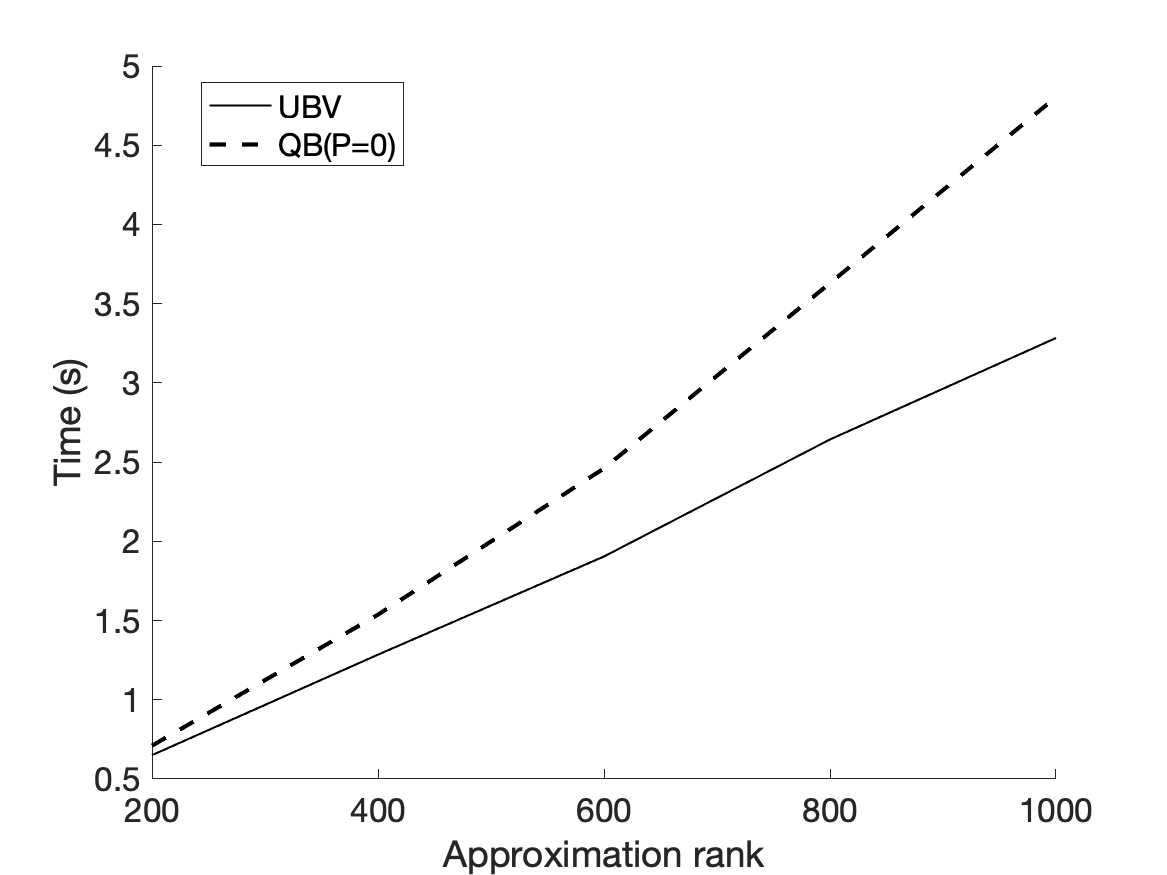}
	\end{minipage}
	\caption{Effects of sparsity (left) and approximation rank (right) on run time. }
	\label{fig:orthogonalize}
\end{figure}

Results for the second and third cases are shown in Figure \ref{fig:orthogonalize}, which confirm our general expectations: for a rectangular matrix with $m > n$, if the matrix is sparse or the approximation rank large then reorthogonalization will take up a larger proportion of the overall cost. Consequently, \texttt{randUBV} will gain a competitive advantage over \texttt{randQB\_EI} due to the fact that it uses one-sided reorthogonalization. This effect will be more pronounced the larger $m$ is compared to $n$, although we found that changing $m$ alone did not have much effect on the relative runtimes of the two algorithms.

\subsection{Block size}
For our third set of test cases, we examine how the choice of block size affects the time and number of iterations required for convergence. We use one synthetic matrix and two real ones: the synthetic matrix is a $4000\times 4000$ matrix whose singular values decrease according to the step function $\sigma_j = 10^{-0.1(\lceil j/30\rceil - 1)}$. Thus each singular value except for the last has multiplicity 30. 

The first real matrix is a dense $3168\times 4752$ matrix, representing the grayscale image of a spruce pine. The second, \texttt{lp\_cre\_b}, comes from a linear programming problem from the SuiteSparse collection \cite{suiteSparse}, and is a $9648\times 77137$ sparse matrix with $260,785$ nonzero elements and at most 9 nonzero elements per column. This second matrix has several sizeable clusters of singular values: for example, $\sigma_{268} \approx 71.10$ and $\sigma_{383} \approx 70.77$. The median relative gap $(\sigma_k-\sigma_{k+1})/\sigma_{k+1}$ among the first 800 singular values is about $8.6\times 10^{-5}$, and the smallest relative gap is about $2.3\times 10^{-8}$. Prior to running the sketching algorithms, both matrices were transposed in order to have more rows than columns. 

\begin{figure}[H]
	\centering
	\begin{minipage}{0.45\textwidth}
		\includegraphics[width=\textwidth]{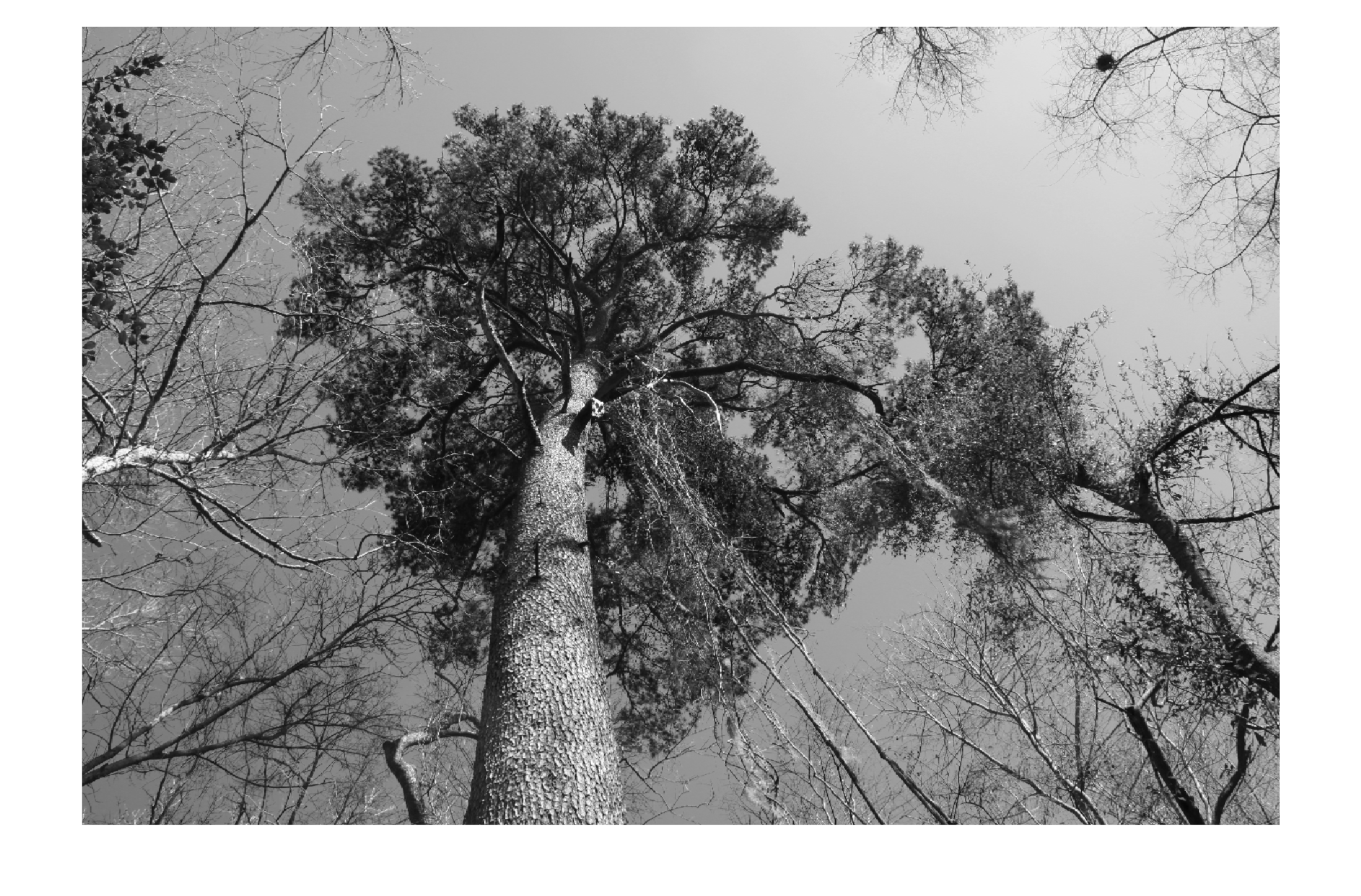}
	\end{minipage}
	\begin{minipage}{0.45\textwidth}
		\includegraphics[width=\textwidth]{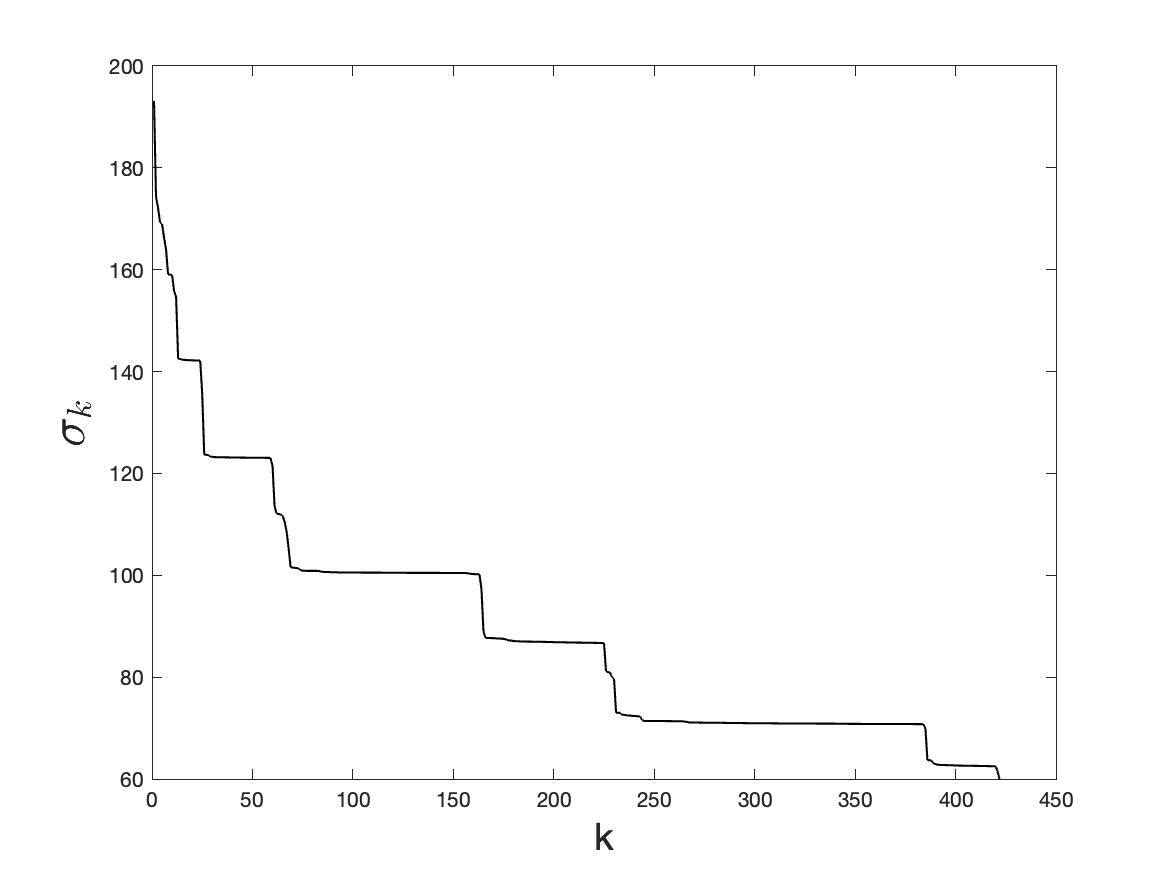}
	\end{minipage}
	\caption{Left: image of {\it pinus glabra}. Right: leading singular values of \texttt{lp\_cre\_b}.}
\end{figure}

We compare \texttt{randUBV} to \texttt{randQB\_EI} with power parameter $p=1$. For both algorithms we approximate the synthetic matrix to a relative error $\tau_{\text{err}} = 0.01$, the grayscale image to a relative error $\tau_{\text{err}}=0.1$, and the SuiteSparse matrix to a relative error $\tau_{\text{err}} = 0.5$. 

Results are shown in Figure \ref{fig:blockTest}. The behavior of \texttt{randQB\_EI} was fairly straightforward: using larger block sizes was more efficient, at least up to the point where the block size was large enough to waste computation by computing $\mat{Q}$ and $\mat{B}$ to a larger rank than necessary. This makes sense because larger block sizes offer more opportunities for using BLAS 3 operations and parallelization. Relatedly, we note that MATLAB's \texttt{svdsketch} function adaptively increases the block size in order to accelerate convergence. 

The behavior of \texttt{randUBV} was very similar to that of \texttt{randQB\_EI} on the grayscale image, but less so on the other two cases.  For the synthetic matrix whose singular values were distributed according to a step function, increasing $b$ from just below the cluster size to just above it led to a sharp drop in both the time and number of iterations required. On the matrix \texttt{lp\_cre\_b}, the optimal block size was near $b=10$ even though the approximation rank was close to constant over all block sizes tested. We speculate that the reason for this is that \texttt{lp\_cre\_b} is both sparse and rectangular, so dense QR operations are a significant portion of the cost of the algorithm. Looking back to the cost of \texttt{randUBV} as shown in \eqref{cost:randubv}, we note that using a smaller block size reduces the cost of performing QR operations on $\mat{U}$. 
%

\begin{figure}[H]
	\begin{subfigure}{\textwidth}
		\centering
		\begin{minipage}{0.45\textwidth}
			\centering
			\includegraphics[width=\textwidth]{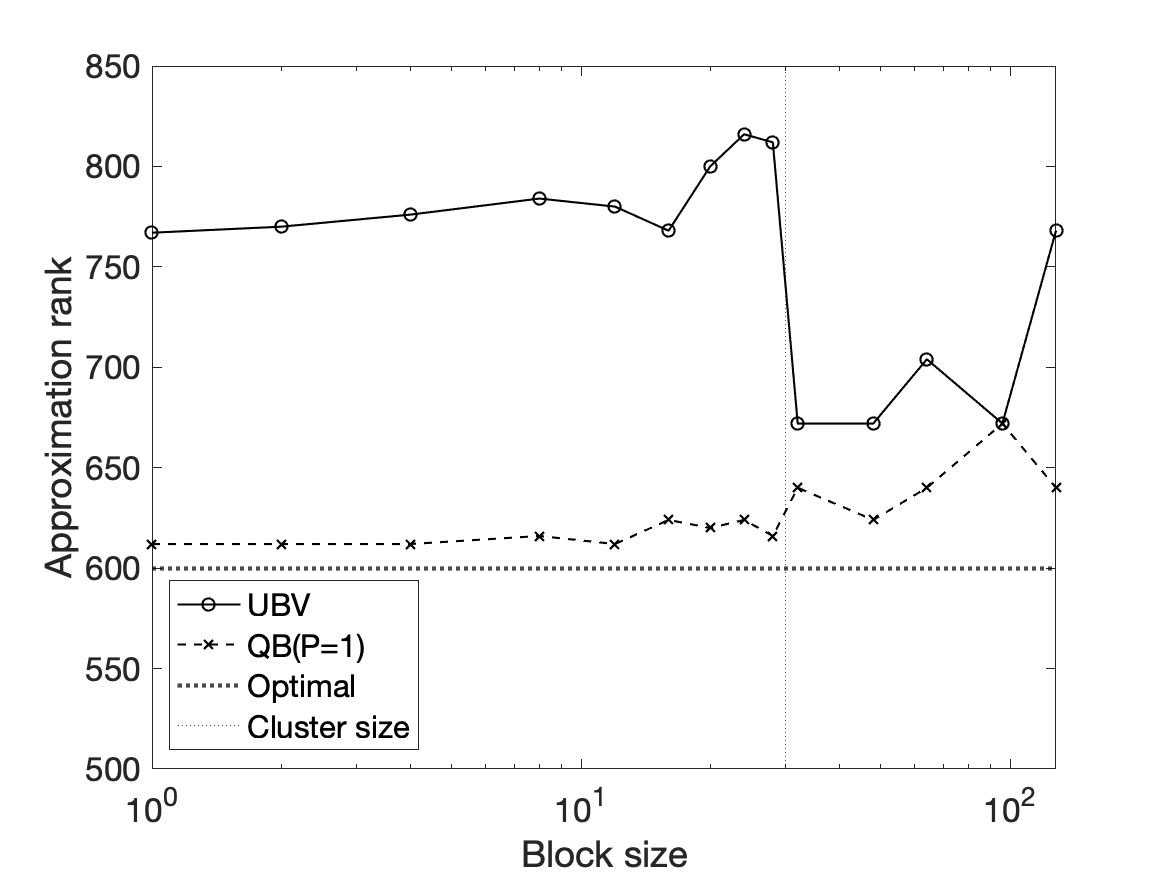}
		\end{minipage}
		\begin{minipage}{0.45\textwidth}
			\centering
			\includegraphics[width=\textwidth]{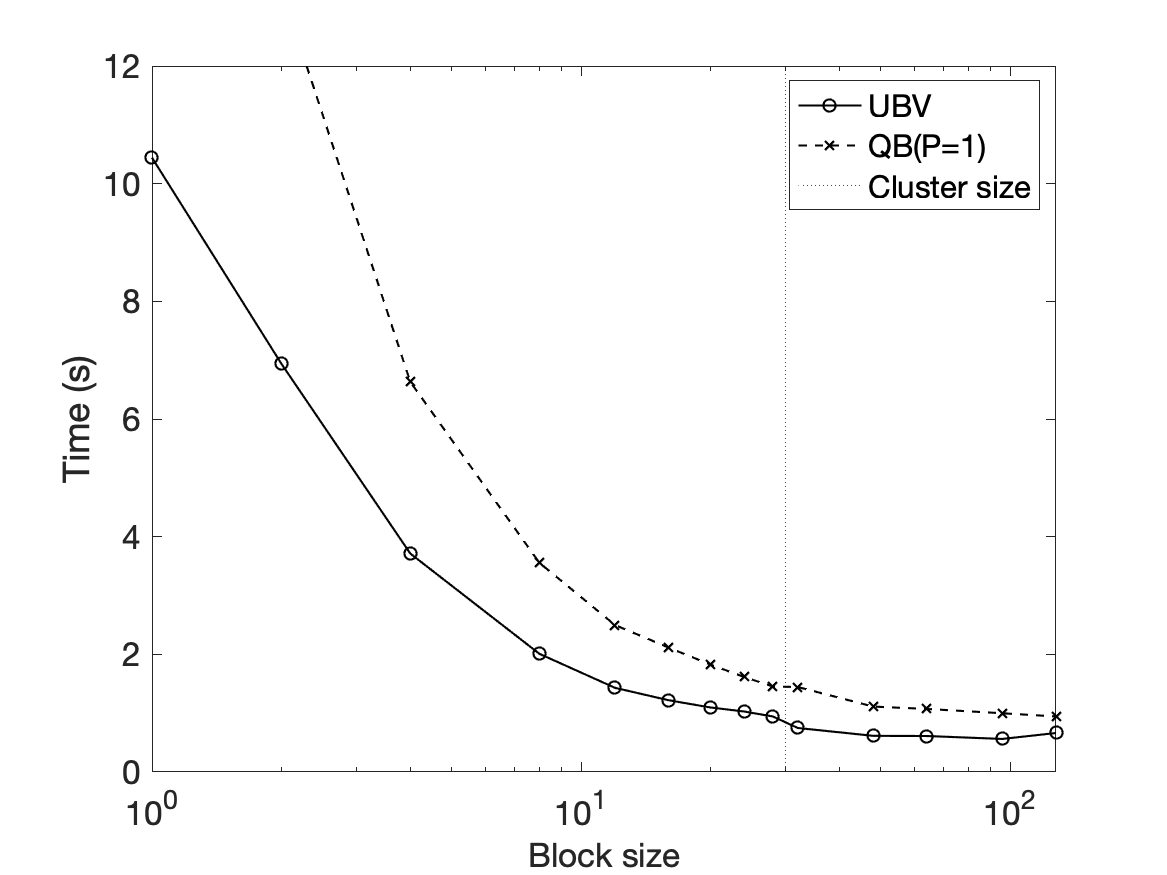}
		\end{minipage}
		\caption{Step function decay. }
	\end{subfigure}
	\begin{subfigure}{\textwidth}
		\centering
		\begin{minipage}{0.45\textwidth}
			\centering
			\includegraphics[width=\textwidth]{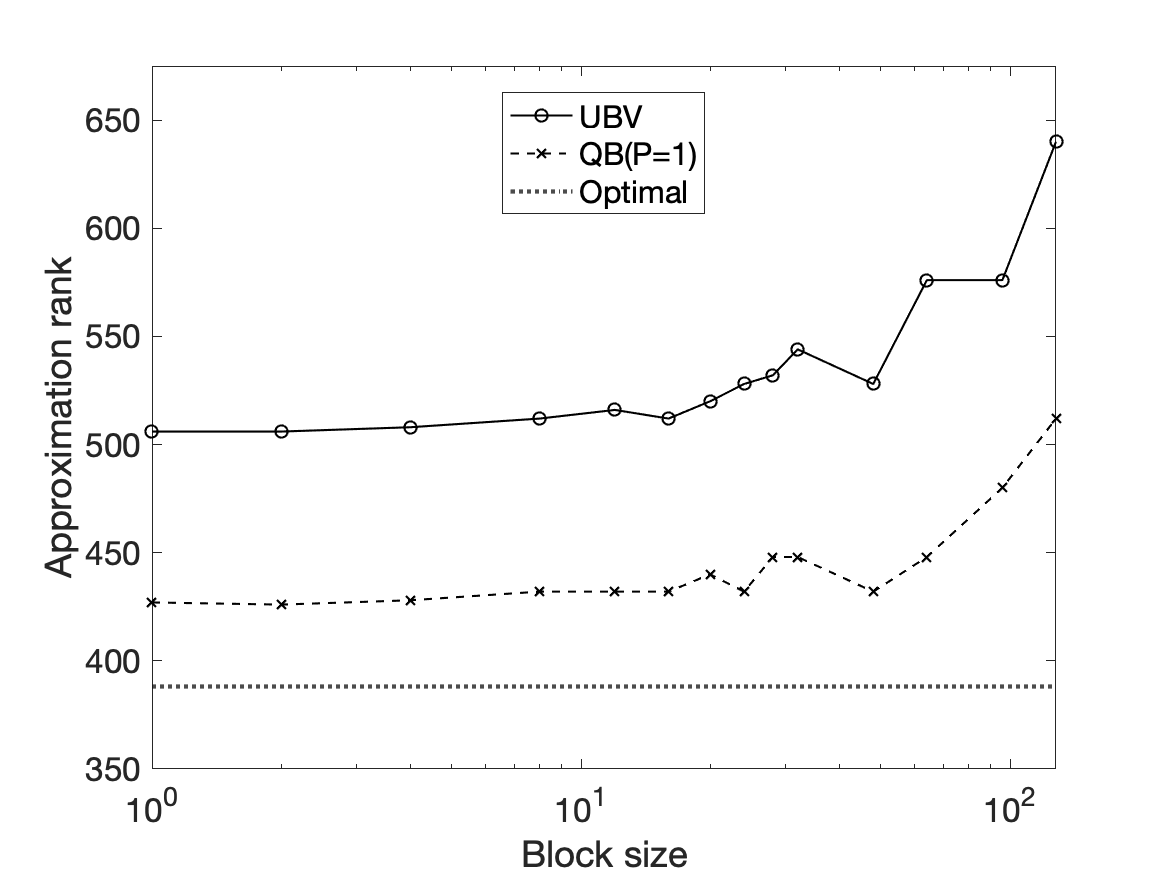}
		\end{minipage}
		\begin{minipage}{0.45\textwidth}
			\centering
			\includegraphics[width=\textwidth]{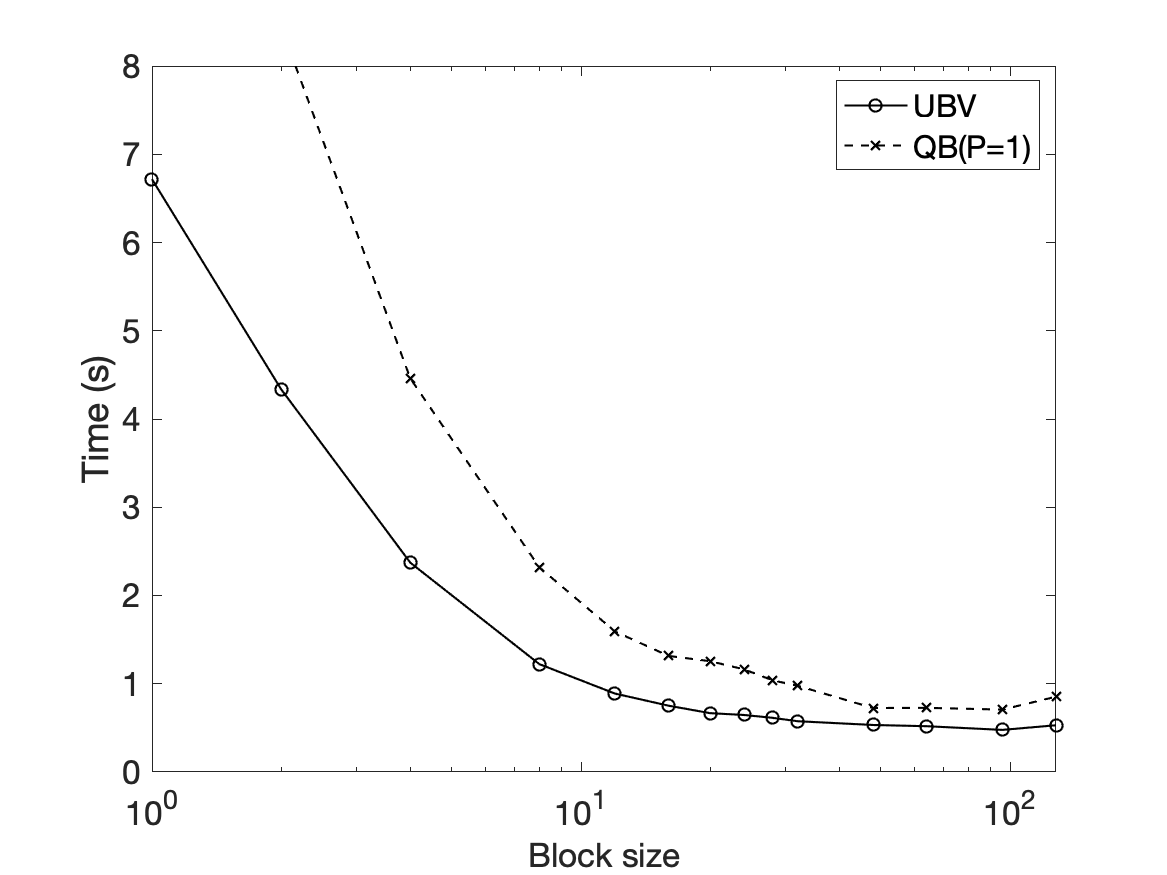}
		\end{minipage}
		\caption{Grayscale image.  }
	\end{subfigure}
	\begin{subfigure}{\textwidth}
		\centering
		\begin{minipage}{0.45\textwidth}
			\centering
			\includegraphics[width=\textwidth]{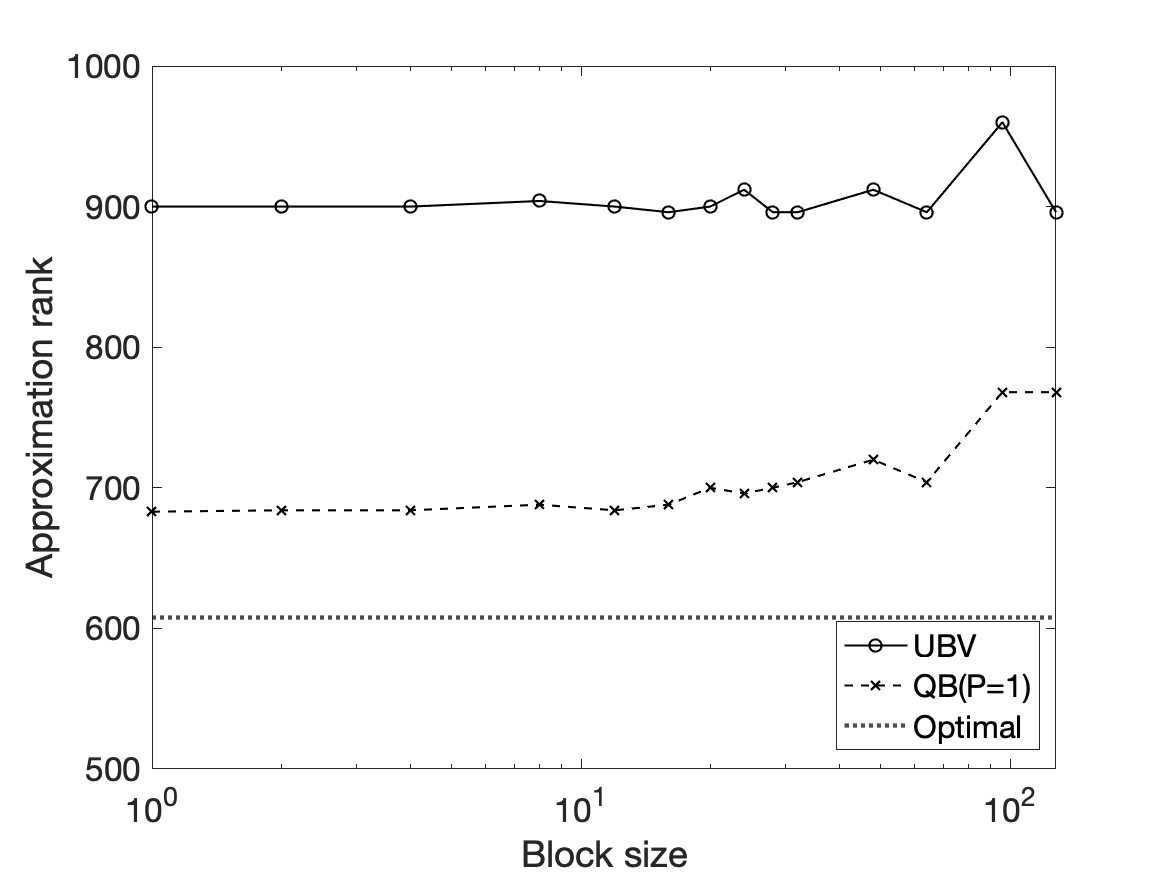}
		\end{minipage}
		\begin{minipage}{0.45\textwidth}
			\centering
			\includegraphics[width=\textwidth]{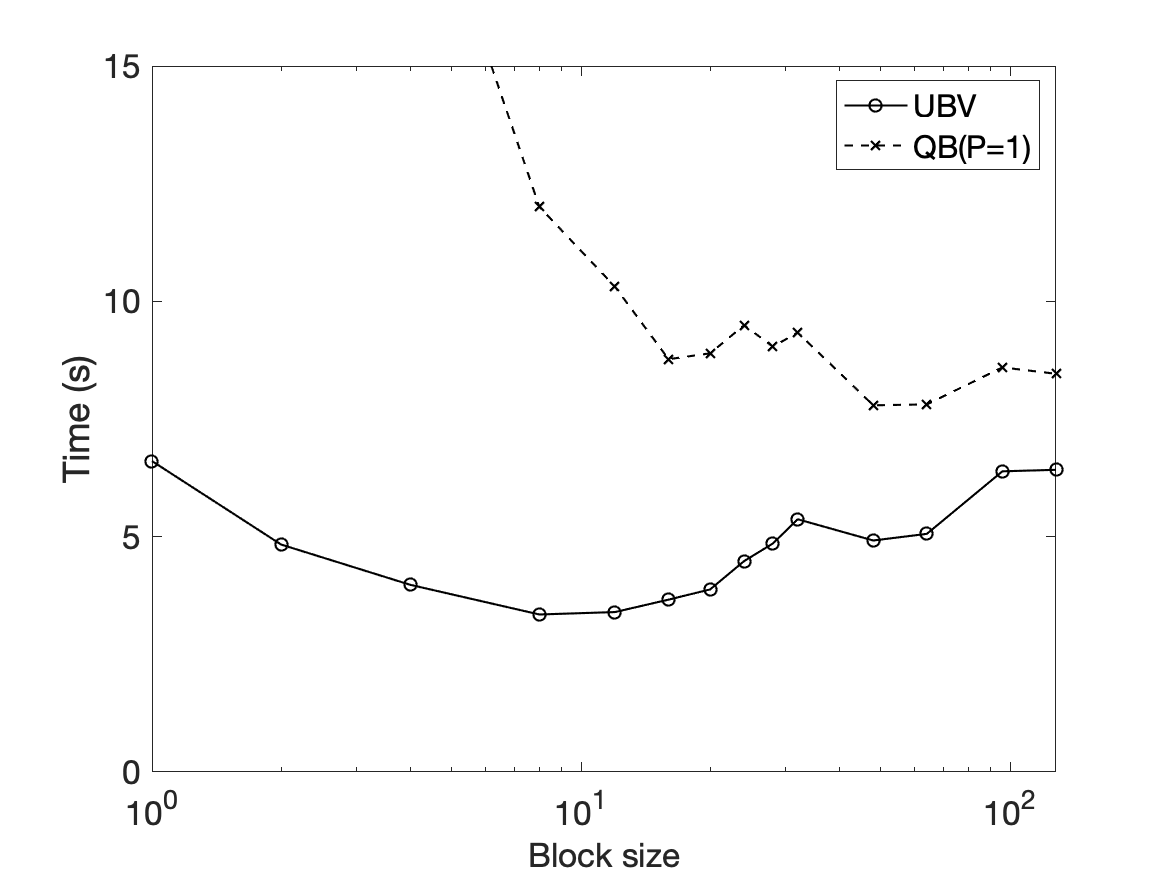}
		\end{minipage}
		\caption{SuiteSparse matrix \texttt{lp\_cre\_b}.  }
	\end{subfigure}
	\caption{Effect of block size the time and number of iterations required for convergence. }
	\label{fig:blockTest}
\end{figure}

\subsection{Stopping tolerance}
In our final set of experiments we examined the effect of choosing a stopping tolerance $\tau_{\text{stop}}$ smaller than the desired approximation error tolerance $\tau_{\text{err}}$, with the conjecture that doing so would allow \texttt{randUBV} to attain significantly better compression rates. We used \texttt{randQB\_EI} with $p=0,1,2$ as a reference for comparison. 

The procedure went as follows: in the first step, each sketching algorithm was run until the Frobenius norm approximation error dropped below a set tolerance $\tau_{\text{stop}}$. In the second step, the SVD of $\mat{B}$ was then computed and truncated as $\mat{B}_r = \mat{U}_r\mat{\Sigma}_r\mat{V}_r\ts$ to the smallest rank such that $\|\mat{A} - \mat{B}_r\|_F \leq \tau_{\text{err}} \|\mat{A}\|_F$, and the singular vectors of $\mat{A}$ computed as $\mat{U}\mat{U}_r$ and $\mat{V}\mat{V}_r$ (or as $\mat{Q}\mat{U}_r$ for \texttt{randQB\_EI}). The time required for each of these two stages was recorded using \texttt{tic} and \texttt{toc}. 

\begin{figure}[H]
	\centering 
	\begin{tabular}{|c|c|c|c|c|c|c|c|}\hline
		Method & $\tau_{\text{stop}}$ & $t_\text{fac}$ & $t_\text{svd}$ & $t_\text{total}$ & $k$ & $r$ \\\hline
		SVD & --  & -- & 13.52 & 13.52 & -- & 388\\\hline
		UBV & 0.1 & 0.68 & 0.08 & 0.76 & 520 & 439\\ \hline
		UBV & 0.09  & 0.87 & 0.11 & 0.98 & 600 & 392\\ \hline
		QB(P=0) & 0.1  & 1.22 & 0.21 & 1.44 & 700 & 663 \\ \hline
		QB(P=1) & 0.1 & 1.12 & 0.09 & 1.22 & 440 & 420\\ \hline
		QB(P=2) & 0.1  & 1.55 & 0.08 & 1.63 & 420 & 398\\ \hline
	\end{tabular}
	\caption{Results for image data with approximation tolerance $\tau_{\text{err}} = 0.1$. }
	\label{fig:photo}
\end{figure}

\subsubsection{Image data}

For the image data, we ran all algorithms to a relative error of $\tau_{\text{stop}} = \tau_{\text{err}} = 0.1$ with block size $b=20$, and for \texttt{randUBV} additionally considered the stricter stopping tolerance $\tau_{\text{stop}} = 0.09$. 

Results are shown in Figure \ref{fig:photo}, with all time reported in seconds. There, $t_\text{fac}$ is the time required for the QB or UBV factorization, $t_\text{svd}$ is the time required to compute the SVD of $\mat{B}$ and the new singular vectors of $\mat{A}$, and $t_\text{total} = t_\text{fac} + t_\text{svd}$. Finally, $k$ is the rank at which the algorithm was terminated, and $r$ the rank to which $\mat{B}$ was truncated. The first line represents the time required to directly compute the SVD of $\mat{A}$ and the optimal truncation rank. 

We observe that \texttt{randUBV} ran faster than \texttt{randQB\_EI} regardless of the value of the power parameter $p$. Even though it required more iterations to converge than \texttt{randQB\_EI} with $p=1$ or $p=2$, it required fewer matrix-vector products with $\mat{A}$ or $\mat{A}\ts$ per iteration. Furthermore, running \texttt{randUBV} to a stopping tolerance that was slighly smaller than the truncation tolerance took somewhat longer but resulted in nearly optimal compression, even superior to subspace iteration with $p=2$. 

\subsubsection{SuiteSparse data}

For the matrix \texttt{lp\_cre\_b} from the SuiteSparse collection, we ran two trials. In the first, we ran all algorithms to the rather modest relative error of $\tau_{\text{stop}} =\tau_{\text{err}} = 0.5$, and for \texttt{randUBV} considered the stricter stopping tolerance $\tau_{\text{stop}} = 0.45$. In the second, we ran the algorithms to the stricter relative error of $\tau_{\text{stop}} =\tau_{\text{err}} = 0.15$, and for \texttt{randUBV} additionally considered $\tau_{\text{stop}} = 0.14$. We used block size $b=50$ for both trials. 

\begin{figure}[H]
	\centering 
	\begin{tabular}{|c|c|c|c|c|c|c|c|}\hline
		Method & $\tau_{\text{stop}}$ & $t_\text{fac}$ & $t_\text{svd}$ & $t_\text{total}$ & $k$ & $r$ \\\hline
		SVD & --  & -- & -- & -- & -- & 608\\\hline
		UBV & 0.5  & 4.69 & 0.93 & 5.62 & 900 & 747\\ \hline
		UBV & 0.45  & 5.68  & 0.99 & 6.67 & 1050 & 627\\ \hline
		QB(P=0) & 0.5 & 8.33 & 8.32 & 16.66 & 1150 & 1123 \\ \hline
		QB(P=1) & 0.5 & 5.16 & 3.69 & 8.85 & 700 & 676\\ \hline
		QB(P=2) & 0.5 & 6.54 & 3.21 & 9.75 & 650 & 627\\ \hline
	\end{tabular}
	\caption{Results for \texttt{lp\_cre\_b} with approximation tolerance $\tau_{\text{err}} = 0.5$.}
	\label{fig:netlib1}
\end{figure}

\begin{figure}[H]
	\centering 
	\begin{tabular}{|c|c|c|c|c|c|c|c|}\hline
		Method & $\tau_{\text{stop}}$  & $t_\text{fac}$ & $t_\text{svd}$ & $t_\text{total}$ & $k$ & $r$ \\\hline
		SVD & --  & -- & -- & -- & -- & 2082\\\hline
		UBV & 0.15  & 21.28 & 12.15 & 33.43 & 2600 & 2293\\ \hline
		UBV & 0.14  & 24.09 & 13.88 & 37.98 & 2700 & 2150\\ \hline
		QB(P=0) & 0.15 & 72.36 & 63.25 & 135.61 & 3600 & 3505 \\ \hline
		QB(P=1) & 0.15 & 38.05 & 22.91 & 60.97 & 2150 & 2147\\ \hline
		QB(P=2) & 0.15 & 48.00 & 21.59 & 69.59 & 2100 & 2100\\ \hline
	\end{tabular}
	\caption{Results for \texttt{lp\_cre\_b} with approximation tolerance $\tau_{\text{err}} = 0.15$.}
	\label{fig:netlib2}
\end{figure}

Results are shown in Figures \ref{fig:netlib1} and \ref{fig:netlib2}, with all time reported in seconds. Due to the size of the matrix $\mat{A}$, we did not attempt to compute its SVD directly but instead found the optimal truncation rank using the precomputed singular values available online \cite{suiteSparse}. 

Once again, \texttt{randUBV} ran faster than its subspace-iteration-based counterpart, and using a slightly smaller stopping tolerance $\tau_{\text{stop}}$ improved the compression ratio without significantly increasing the runtime. The iteration $k$ at which \texttt{randUBV} terminated was significantly smaller than it was for \texttt{randQB\_EI} with $p=0$, but significantly larger than for \texttt{randQB\_EI} with $p=1$ or $p=2$ (perhaps in part due to the singular value clusters). 

It should be noted that the matrix $\mat{A}$ in question is quite sparse with only about $0.03\%$ of its entries nonzero, and fairly skinny with $m\approx 8n$. It is therefore worth exploring whether \texttt{randQB\_EI} might save time on reorthogonalization costs if performed on $\mat{A}\ts$ instead. We re-ran the experiment for $\tau_{\text{err}} = 0.15$, and found that while the factorization time $t_\text{fac}$ did not change much, the second step $t_\text{svd}$ took around twice as long due to the matrix $\mat{B}$ being $k\times m$ rather than $k\times n$.

\section{Conclusions}
\label{sec:conclusions}

We have proposed a randomized algorithm \texttt{randUBV} that takes a matrix $\mat{A}$ and uses block Lanczos bidiagonalization to find an approximation of the form $\mat{UBV}\ts$, where $\mat{U}$ and $\mat{V}$ each have orthonormal columns in exact arithmetic and $\mat{B}$ is a block bidiagonal matrix. For square matrices it costs approximately the same per iteration as \texttt{randQB}-type methods run with power parameter $p=0$ while having better convergence properties. On rectangular matrices, it exploits one-sided reorthognalization to run faster without much degrading the accuracy of the error estimator. Numerical experiments suggest that \texttt{randUBV} is generally competitive with existing \texttt{randUBV}-type methods, at least as long as the problem is not so large that it becomes important to minimize the number of passes over $\mat{A}$. 

A few avenues for future exploration are suggested. First and most importantly, roundoff error allows block Lanczos methods to handle repeated singular values, which they would be unable to do in exact arithmetic. This fact has been known for decades, but we are not currently aware of any rigorous convergence bounds that account for finite precision. Second, reinflation or any more general method for adaptively changing the block size $b$ will make the span of $\mat{V}$ a sum of Krylov spaces of different dimensions. We are not aware of any convergence results that cover this more general setting. 

It is also worth exploring just how much the block Lanczos method benefits from oversampling. We have observed that running \texttt{randUBV} for a few more iterations than necessary can result in near-optimal compression, but it would be worthwhile to turn the convergence results of e.g. \cite{yuan2018superlinear} into practical guidance on how many more iterations are necessary. 


Finally, the behavior of $\mat{U}$ when using one-sided reorthogonalization merits further study. We generally found that when using a larger stopping tolerance $\tau$ the columns of $\mat{U}$ remained closer to orthonormal. It would be highly desirable to obtain a rigorous result establishing that one-sided reorthogonalization is safe as long as only a rough approximation is required, but we leave this goal for a future work.

MATLAB code is available at \url{https://github.com/erhallma/randUBV}, including our main algorithm \texttt{randUBV} as well as code used to reproduce the figures used in this paper.


\section*{Acknowledgments}
The author would like to thank Ilse Ipsen and Arvind Saibaba for their helpful comments on an earlier draft of this paper. 

\bibliographystyle{siamplain}
\bibliography{references}
\end{document}